\documentclass[a4paper,12pt]{article}

\usepackage{mymacros}

\newcommand{\op}{\mathrm{op}}

\title{Dimer models and crepant resolutions}
\author{Akira Ishii and Kazushi Ueda}
\date{}
\begin{document}

\maketitle

\begin{abstract}
We study variations of tautological bundles
on moduli spaces
of representations
of quivers with relations
associated with dimer models
under a change of stability parameters.
We prove that if the tautological bundle
induces a derived equivalence
for some stability parameter,
then the same holds for any generic stability parameter,
and any projective crepant resolution
%of the corresponding Gorenstein toric singularity
can be obtained as the moduli space
%of quiver representations
for some stability parameter.
%Our proof is based on the argument
%of Craw and Ishii \cite{Craw-Ishii}, and
%does not depend on the result
%of Bridgeland, King and Reid
%\cite{Bridgeland-King-Reid}.
This result is used
in \cite{Ishii-Ueda_DMSMC}
to prove the abelian McKay correspondence
%as a derived equivalence
without using the result
of Bridgeland, King and Reid
\cite{Bridgeland-King-Reid}.
\end{abstract}

\section{Introduction}

In this paper,
we study
\begin{itemize}
 \item
variations of geometric invariant theory quotients,
 \item
tautological constructions on moduli spaces, and
 \item
moduli interpretations of toric varieties
\end{itemize}
for smooth toric Calabi-Yau 3-folds.

The theory of variations of geometric invariant theory quotients of a normal variety
\cite{Dolgachev-Hu,Thaddeus_flip}
deals with the dependence of geometric invariant theory quotients
on the choice of stability parameters
(i.e. the choice of a $G$-linearization on an ample line bundle).
The space of stability parameters is divided into chambers
by real codimension one walls,
and the quotient changes as one moves from one chamber to another.
If the generic point is stable for any stability parameter,
then different quotients are related by birational transformations
such as blow-ups and flips.
A prototypical example is a toric variety
constructed as the quotient of an affine space
by an action of a torus.

Tautological constructions on moduli spaces
are fundamental tools to study geometry of moduli spaces.
The moduli space that we study in this paper
is the moduli space $\scM_\theta$
of stable representations of a quiver $\Gamma$ with relations,
which is constructed by King \cite{King}
using geometric invariant theory.
Here $\theta$ is a stability parameter,
and we write the tautological bundle on $\scM_\theta$
as $\scE = \bigoplus_v \scL_v$
where $v$ runs over the set of vertices of the quiver.
The tautological bundle exists if the dimension vector
is a primitive vector \cite[Proposition 5.3]{King}.
The moduli space $\scM_{\theta}$ is an open subscheme
of the coarse moduli scheme $\overline{\scM_\theta}$
parametrizing S-equivalence classes of $\theta$-semistable representations,
which is projective over the affine scheme $\overline{\scM_0}$ defined for the
stability parameter $0$.
We consider representations
with dimension vector $(1,\dots,1)$,
so that each $\scL_v$ is a line bundle.
The stability parameter $\theta$ is said to be {\em generic}
if any semi-stable objects are stable.
Genericity of $\theta$ means that $\scM_\theta=\overline{\scM_\theta}$.

Assume that $\scM_\theta$ is smooth, and
consider the following two conditions:
\begin{itemize}
 \item[(\bfT)]
The tautological bundle $\scE$
on the moduli space $\scM_\theta$ is a tilting object.
 \item[(\bfE)]
The universal morphism
$
 \bC \Gamma \to \End \scE
$
is an isomorphism.
\end{itemize}
According to Morita theory for derived category
\cite{Bondal_RAACS, Rickard},
the conditions (\bfT)+(\bfE) imply that the functor
\begin{equation} \label{eq:der_equiv}
 \Phi(-)
   = \bR \Gamma \lb \scE \otimes - \rb
   : D^b \coh \scM_\theta \to D^b \module \bC \Gamma
\end{equation}
is an equivalence of triangulated categories.

%King \cite{King_tilt} conjectured that
%every smooth compact toric variety admits a tilting object
%which is a direct sum of line bundles.
%If true, this will provide a moduli interpretation
%of any smooth projective toric variety
%\cite{Bergman-Proudfoot}.
%Unfortunately,
%this conjecture was disproved by
%Hille and Perling
%\cite{Hille-Perling}
%for a Hirzebruch surface blown-up three times.
%Borisov and Hua \cite{Borisov-Hua}
%suggested that the conjecture might be true
%for smooth Fano toric Deligne-Mumford stacks
%and proved it in dimension two.
%However,
%Efimov \cite{Efimov_MLECLB} has shown that
%the conjecture is false even for toric Fano manifolds
%in higher dimensions, let alone toric Fano stacks.
%On the other hand,
%a fine moduli interpretation of projective toric varieties
%are obtained by Craw and Smith \cite{Craw-Smith}
%and generalized to stacks by Abdelgadir \cite{Abdelgadir:2012}.
%Their construction neither gives a tilting object
%nor a derived equivalence.

In this paper,
we study the special case
where the quiver $\Gamma$ with relations
comes from a dimer model.
A {\em dimer model} is
a bicolored graph on a real 2-torus,
%giving a polygon division of the torus.
which encodes the information
of a quiver with relations.
%It is invented by string theorists
%to study supersymmetric quiver gauge theories
%in four dimensions
%(see e.g. \cite{Kennaway_BT} and references therein).
%\cite{Franco-Hanany-Martelli-Sparks-Vegh-Wecht_GTTGBT,
%Franco-Hanany-Vegh-Wecht-Kennaway_BDQGT,
%%Franco-Vegh_MSGTDM,
%Hanany-Herzog-Vegh_BTEC,
%Hanany-Kennaway_DMTD,
%Hanany-Vegh}.
If the dimer model is {\em non-degenerate},
%satisfies a certain non-degeneracy condition,
then the moduli space $\scM_{\theta}$
of stable representations of the corresponding quiver
with dimension vector $(1, \dots, 1)$
with respect to a generic stability parameter $\theta$
%in the sense of King \cite{King}
is a smooth toric Calabi-Yau 3-fold,
% (i.e., if every edge is contained
% in at least one perfect matching)
which is a crepant resolution of a Gorenstein affine
toric variety $X$
\cite{Ishii-Ueda_08}.

%As a partial converse to this, we first prove the following:
%
%\begin{proposition} \label{pr:non-degenerate}
%Let $G$ be a dimer model.
%If the moduli space $\scM_\theta$ is three-dimensional and is a crepant resolution of
%the Gorenstein affine toric variety $X$
%for some generic $\theta$ and if the conditions \textup{(\bfT)+(\bfE)} holds,
%then $G$ is non-degenerate.
%\end{proposition}

The main result in this paper
is the following:

\begin{theorem} \label{th:change_of_stability}
Let $G$ be a non-degenerate dimer model.
If the conditions \textup{(\bfT)+(\bfE)} hold
for some generic stability parameter $\theta$,
then the same hold
for any generic stability parameter.
\end{theorem}

Although Theorem \ref{th:change_of_stability} follows
from the result of Bridgeland, King and Reid
\cite{Bridgeland-King-Reid, Van_den_Bergh_NCR},
we give an alternative proof
based on an explicit description
of variation of moduli spaces
under a change of stability parameters.
Such a description
for moduli spaces of the McKay quiver
associated with a finite small abelian subgroup $A$ of $\SL_3(\bC)$
is given %by Craw and the first author
in \cite{Craw-Ishii}
using the result of Bridgeland, King and Reid.
Our proof of Theorem \ref{th:change_of_stability}
works for general dimer models,
and does not rely on the result of Bridgeland, King and Reid.

The proof of Theorem \ref{th:change_of_stability} also gives
a generalization
of the main result of %Craw and Ishii
\cite{Craw-Ishii} to dimer models:

\begin{corollary} \label{cr:crepant_resolution}
Let $G$ be a dimer model
satisfying conditions $(\bfT)+(\bfE)$
for some (and hence any) generic
stability parameter.
Then any projective crepant resolution
of the Gorenstein affine toric variety $X$
associated with the dimer model $G$
is obtained as the moduli space $\scM_\theta$
for some generic stability parameter $\theta$.
\end{corollary}

Combined with \cite[Theorem 1.4]{Ishii-Ueda_DMSMC},
this provides a fine moduli interpretation
of any projective crepant resolution
of any 3-dimensional Gorenstein affine toric variety $X$.

Theorem \ref{th:change_of_stability} is used
in \cite{Ishii-Ueda_DMSMC}
to circumvent the use
of the result of Bridgeland, King and Reid
\cite{Bridgeland-King-Reid}
in the proof of the main result.
Since the main result of \cite{Ishii-Ueda_DMSMC}
contains the abelian McKay correspondence
as a special case,
this gives a new proof
of the abelian McKay correspondence
as a derived equivalence.
%This also makes the proof
%%the proof of main results of \cite{Ishii-Ueda_DMSMC}
%logically independent
%from the results of Mozgovoy and Reineke
%\cite{Mozgovoy-Reineke},
%Davison \cite{Davison},
%Broomhead \cite{Broomhead}, and
%Bocklandt \cite{Bocklandt_CYAWQP, Bocklandt_CCDM}.
This paper is separated from \cite{Ishii-Ueda_DMSMC}
since the technique used in this paper
%which follows Craw and Ishii \cite{Craw-Ishii} closely,
is different
from the main line of the argument
in \cite{Ishii-Ueda_DMSMC},
and only the statement
of Theorem \ref{th:change_of_stability}
is needed in \cite{Ishii-Ueda_DMSMC}.
Many of the arguments in this paper are parallel to
those in \cite{Craw-Ishii};
the main difference is that our argument works
for general dimer models and
does not depend
on the result of Bridgeland, King and Reid
\cite{Bridgeland-King-Reid}.

This paper is organized as follows:
In Section \ref{sc:definitions},
we recall basic definitions on dimer models
and associated quivers.
In Section \ref{sc:moduli},
we recall the definition of the moduli space
of stable representations of quivers.
In Section \ref{sc:GIT},
we recall the construction of the moduli space
due to King \cite{King}.
It is based on geometric invariant theory,
and resulting moduli space
comes naturally with an ample line bundle.
%
%In Section \ref{sc:prop1.1},
%we give a proof of Proposition \ref{pr:non-degenerate}.
%
In Section \ref{sc:non-degenerate},
we recall the main result of \cite{Ishii-Ueda_08}
which deals with the case
when the quiver with relations comes from a non-degenerate dimer model.
In Section \ref{sc:walls_and_chambers},
we describe walls and chambers
in the space of stability parameters.
In Section \ref{sc:equivalence},
we recall the definition of a tilting object
and its relation with the derived equivalence.
In Section \ref{sc:stability_parameters},
we give a natural identification
between the space of stability parameters
and a subspace
of the Grothendieck group
of finitely-generated modules
over the path algebra of the quiver.
In Section \ref{sc:KvsPic},
we assume the derived equivalence and
show that a primitive contraction of the moduli space occurs
on the walls in the space of stability parameters.
In Section \ref{sc:variation_of_moduli},
we study destabilizing sequences.
%
%In Section \ref{sc:Quot},
%we study the obstruction class
%of the Quot scheme for quiver representations
%needed later in Section \ref{sc:variation_of_bundles}.
%
In Section \ref{sc:variation_of_bundles},
we study variation of tautological bundles and
prove Theorem \ref{th:change_of_stability}.
In Section \ref{sc:crepant_resolution},
we give the proof of Corollary \ref{cr:crepant_resolution}.

{\bf Acknowledgment}:
A.~I. is supported by Grant-in-Aid for Scientific Research (No.18540034).
K.~U. is supported by Grant-in-Aid for Young Scientists (No.24740043).

\section{Dimer models and quivers}
 \label{sc:definitions}
 
 Let $T = \bR^2 / \bZ^2$
be a real two-torus
equipped with an orientation.
A {\em bicolored graph} on $T$
consists of
\begin{itemize}
 \item a finite set $B \subset T$ of black nodes,
 \item a finite set $W \subset T$ of white nodes, and
 \item a finite set $E$ of edges,
       consisting of embedded closed intervals $e$ on $T$
       such that one boundary of $e$ belongs to $B$
       and the other boundary belongs to $W$.
       We assume that two edges intersect
       only at the boundaries.
\end{itemize}
A {\em face} of a graph is a connected component
of $T \setminus \cup_{e \in E} e$.
A bicolored graph $G$ on $T$ is called a {\em dimer model}
if every face is simply-connected and there is no univalent node.

A {\em quiver} consists of
\begin{itemize}
 \item a set $V$ of vertices,
 \item a set $A$ of arrows, and
 \item two maps $s, t: A \to V$ from $A$ to $V$.
\end{itemize}
For an arrow $a \in A$,
the vertices $s(a)$ and $t(a)$
are said to be the {\em source}
and the {\em target} of $a$
respectively.
A {\em path} on a quiver
is an ordered set of arrows
$(a_n, a_{n-1}, \dots, a_{1})$
such that $s(a_{i+1}) = t(a_i)$
for $i=1, \dots, n-1$.
We also allow for a path of length zero,
starting and ending at the same vertex.
The {\em path algebra} $\bC Q$
of a quiver $Q = (V, A, s, t)$
is the algebra
spanned by the set of paths
as a vector space,
and the multiplication is defined
by the concatenation of paths;
$$
 (b_m, \dots, b_1) \cdot (a_n, \dots, a_1)
  = \begin{cases}
     (b_m, \dots, b_1, a_n, \dots, a_1) & s(b_1) = t(a_n), \\
      0 & \text{otherwise}.
    \end{cases}
$$
A {\em quiver with relations}
is a pair of a quiver
and a two-sided ideal $\scI$
of its path algebra.
For a quiver $\Gamma = (Q, \scI)$
with relations,
its path algebra $\bC \Gamma$ is defined as
the quotient algebra $\bC Q / \scI$.
Two paths $a$ and $b$ are said to be {\em equivalent}
if they give the same element in $\bC \Gamma$.

A dimer model $(B, W, E)$ encodes
the information of a quiver
$\Gamma = (V, A, s, t, \scI)$
with relations
in the following way:
The set $V$ of vertices
is the set of faces
%the set of connected components
%of the complement
%$
% T \setminus (\bigcup_{e \in E} e),
%$
and
the set $A$ of arrows
is the set $E$ of edges of the graph.
The directions of the arrows are determined
by the colors of the nodes of the graph,
so that the white node $w \in W$ is on the right
of the arrow.
In other words,
the quiver is the dual graph of the dimer model
equipped with an orientation given by
rotating the white-to-black flow on the edges of the dimer model
by minus 90 degrees.

The relations of the quiver are described as follows:
For an arrow $a \in A$,
there exist two paths $p_+(a)$
and $p_-(a)$
from $t(a)$ to $s(a)$,
the former going around the white node
connected to $a \in E = A$ clockwise
and the latter going around the black node
connected to $a$ counterclockwise.
Then the ideal $\scI$
of the path algebra is
generated by $p_+(a) - p_-(a)$
for all $a \in A$.

A {\em small cycle} on a quiver coming from a dimer model
is the product of arrows surrounding only a single node
of the dimer model.
%A path $p$ is said to be {\em minimal}
%if it is not equivalent to a path containing a small cycle.
%A path $p$ is said to be {\em minimum}
%if any path from $s(p)$ to $t(p)$ homotopic to $p$ is equivalent
%to the product of $p$ and a power of a small cycle.
%For a pair of vertices of the quiver,
%a minimum path from one vertex to another may not exist,
%and will always be minimal when it exists.
%
Small cycles starting from a fixed vertex
are equivalent to each other.
Hence the sum $\omega$ of small cycles
over the set of vertices is a well-defined
element of the path algebra,
%For any arrow $a$,
%the small cycles $\omega_{s(a)}$ and $\omega_{t(a)}$
%starting from the source $s(a)$ and the target $t(a)$
%of $a$ respectively satisfies
%$$
% a \omega_{s(a)} = \omega_{t(a)} a.
%$$
%If follows that $\omega$
which belongs to the center of the path algebra.
%and there is the universal map
%$$
% \bC \Gamma \to \bC \Gamma[\omega^{-1}]
%$$
%into the localization of the path algebra
%by the multiplicative subset generated by $\omega$.
%Two paths $a$ and $b$ are said to be {\em weakly equivalent}
%if they are mapped to the same element
%in $\bC \Gamma[\omega^{-1}]$,
%i.e., there is an integer $i \ge 0$ such that
%$a \omega^i = b \omega^i$ in $\bC\Gamma$.

%A {\em zigzag path} is a path on a dimer model
%which makes a maximum turn to the right on a white node
%and to the left on a black node.
%Note that it is not a path on a quiver.
%We assume that a zigzag path does not have an endpoint,
%so that we can regard a zigzag path
%as a sequence $(e_i)$ of edges $e_i$ parameterized by $i \in \bZ$,
%up to translations of $i$.
%Figure \ref{fg:zigzag} shows an example
%of a part of a dimer model
%and a zigzag path on it.
%
%\begin{figure}[htbp]
%\centering
%\input{zigzag.pst}
%\caption{A zigzag path}
%\label{fg:zigzag}
%\end{figure}

A {\em perfect matching}
(or a {\em dimer configuration})
on a bicolored graph $G = (B, W, E)$
is a subset $D$ of $E$
such that for any node $v \in B \cup W$,
there is a unique edge $e \in D$
connected to $v$.

A dimer model is said to be {\em non-degenerate}
if for any edge $e \in E$,
there is a perfect matching $D$
such that $e \in D$.

\section{Moduli spaces of quiver representations}
 \label{sc:moduli}

A {\em representation} of a quiver
$
\Gamma = (V, A, s, t, \scI)
$
with relations is a module
over the path algebra $\bC  \Gamma$.
In other words,
a representation of $\Gamma$ is a collection
$M = ((M_v)_{v \in V}, (\psi_a)_{a \in A})$
of vector spaces $M_v$ for $v \in V$
and linear maps $\psi_a : M_{s(a)} \to M_{t(a)}$
for $a \in A$ satisfying relations in $\scI$.

The {\em dimension vector}
of a finite-dimensional representation
$((M_v)_{v \in V}, (\psi_a)_{a \in A})$
is defined as $(\dim M_v)_{v \in V} \in \bZ^V$.
The {\em support} of a representation
is the set of vertices $v \in V$
such that $\dim M_v \ne 0$.

A {\em stability parameter} $\theta$
for a fixed dimension vector $d$ is
an element of
\begin{equation} \label{eq:Theta}
 \Theta = \{ \theta \in \Hom (\bZ^V, \bZ) \mid
   \theta( d ) = 0 \}.
\end{equation}
%where $\chi_V = (1, \ldots, 1) \in \bZ^V$ is
%the characteristic function of $V$.

A $\bC \Gamma$-module $M$
with dimension vector $d$
is said to be {\em $\theta$-stable}
if for any non-trivial submodule
$0 \ne N \subsetneq M$,
one has $\theta(N) > \theta(M)=0$.
The module $M$ is {\em $\theta$-semistable}
if $\theta(N) \ge \theta(M)$ holds
instead of $\theta(N) > \theta(M)$.
A stability parameter $\theta$ is said to be {\it generic}
with respect to a fixed dimension vector
if semistability implies stability.
A pair $(\theta, \eta)$ of stability parameters
are said to define the same stability condition
with respect to $d$
if a $\bC \Gamma$-module $M$
with dimension vector $d$
is $\theta$-stable if and only if it is $\eta$-stable.
This stability condition is introduced by King \cite{King}
to construct the moduli space $\scM_{\theta}$
representing (the sheafification of) the functor
$$
\begin{array}{ccc}
 (\scS ch) & \to & (\scS et) \\
  \vin & & \vin \\
  T & \mapsto & (\text{a flat family over $T$
  of $\theta$-stable representations of $\Gamma$} )/ \sim.
\end{array}
$$
Here, a {\em flat family} of representations of $\Gamma$
over $T$
consists of a collection $(\scL_v)_{v \in V}$
of vector bundles on $T$
for each vertex $v$ of $\Gamma$ and
a collection $(\phi_a)_{a \in A}$ of morphisms
$\phi_a : \scL_{s(a)} \to \scL_{t(a)}$
for each arrow $a$ of $\Gamma$
satisfying the relations $\scI$ of $\Gamma$.
Two families $((\scL_v)_{v \in V}, (\phi_a)_{a \in A})$
and $((\scL_v')_{v \in V}, (\phi_a')_{a \in A})$
are {\em equivalent}
if there are a line bundle $\scL$
and a collection $(\psi_v)_{v \in V}$
of isomorphisms $\psi_v : \scL_v \simto \scL_v' \otimes \scL$
of vector bundles such that the diagram
$$
\begin{CD}
 \scL_{s(a)} @>{\phi_a}>> \scL_{t(a)} \\
 @V{\psi_{s(a)}}VV @VV{\psi_{t(a)}}V \\
 \scL_{s(a)}' \otimes \scL @>{\phi_a \otimes 1_{\scL}}>> \scL_{t(a)} \otimes \scL
\end{CD}
$$
commutes for all $a \in A$.
In our case,
the dimension vector $(1, \dots, 1)$ is primitive,
so that we do not have to sheafify the functor,
and there is a universal family over the moduli space.
The bundles $\scL_v$ in the universal family are called
the {\it tautological bundles}.
We have a freedom of simultaneously tensoring a line bundle
in the choice of the tautological bundles.
In this paper,
we choose a vertex $v_0$ of the quiver $\Gamma$
once and for all,
and normalize the tautological bundle
so that $\scL_{v_0}$ is the trivial bundle.
The moduli space of $\theta$-stable $\bC\Gamma$-modules
with dimension vector $(1, \dots, 1)$
is denoted by $\scM_\theta$.
On the other hand,
the moduli space $\scMbar_\theta$
of $\theta$-semistable modules
does not represent the moduli functor,
but parametrizes S-equivalence classes of
$\theta$-semistable modules.

\section{The ample line bundle on the moduli space}
 \label{sc:GIT}

The construction of the moduli spaces
$\scM_\theta$ and $\scMbar_\theta$
by King \cite{King} is based on geometric invariant theory,
and comes with a natural ample line bundle.
Here we recall the construction in the case where the dimension vector
is $(1, \dots ,1)$.
Fix one-dimensional $\bC$-vector spaces
$M_v=\bC$ for each vertex $v \in V$ and
consider the affine scheme
$$
 \scN = \{ (\psi_a)_{a \in A} \in \bC^A \mid
  \text{$(\psi_a)_{a \in A}$ satisfies the relations in $\scI$} \}.
$$
Let
$$
 \scG
  = (\bCx)^V
%  = \{ (g_v)_{v \in V} \mid g_v \in \bCx \} / \bCx
$$
be the group of gauge transformations,
which acts on $\scN$ by
$$
 (g_v)_{v \in V} : 
 ( \psi_a )_{a \in A}
   \mapsto
 ( \psi_a' = g_{t(a)} \cdot \psi_a \cdot g_{s(a)}^{-1} )_{a \in A}.
$$
The action of the diagonal subgroup $\bCx \subset \scG$ is trivial,
and the quotient group will be denoted by
$\bP \scG = \scG / \bCx$.
A stability parameter $\theta \in \Theta$ defines the character
$$
\begin{array}{cccc}
 \chitilde : & \scG & \to & \bCx \\
  & \vin & & \vin \\
  & (g_v)_{v \in V} & \mapsto & \prod_{v \in V} g_v^{\theta(v)},
\end{array}
$$
which descends to the character
$\chi : \bP \scG \to \bCx$
by the condition $\theta(1, \dots, 1) = 0$.
%in \eqref{eq:Theta}.
Let $\bC[\scN]_{\chi^n}^\scG$ be the space
of semi-invariant functions of weight $\chi^n$ on $\scN$,
and consider the graded ring
$
 R_\chi = \bigoplus_{n=0}^\infty \bC[\scN]_{\chi^n}^\scG.
$
The moduli space $\scMbar_\theta$ is obtained as
$
 \scMbar_\theta
  = \Proj R_\chi,
%  = \Proj \bigoplus_{n=0}^\infty \bC[\scN]_{\chi^n}^\scG,
$
and the moduli space $\scM_\theta$ is its open subscheme
parametrizing stable orbits.
It follows from the construction
that the moduli space $\scM_\theta$ comes
with an ample line bundle $\scO_{\scM_\theta}(1)$.

A choice of a splitting $s : \bP \scG \to \scG$ defines
tautological bundles $\scL_v$ for $v \in V$
as the line bundle associated with the character
$\pi_v \circ s : \bP \scG \to \bCx$ of $\bP \scG$,
where
$\pi_v : \scG \to \bCx$
is the projection
$
(g_w)_{w \in V} \mapsto g_v
$
to the $v$-th component.
Our choice for the splitting $s$ is such that
$s(\alpha)_{v_0} = 1$ for the fixed vertex $v_0$
and any $\alpha \in \bCx$.
This gives the normalization 
of the tautological bundles
such that $\scL_{v_0}$ is the trivial bundle.

Let $L_\theta : \Theta \to \Pic \scM_\theta$ be the map
defined by
\begin{equation} \label{eq:L_theta}
\begin{array}{cccc}
 L_\theta : & \Theta & \to & \Pic \scM_\theta \\
  & \vin & & \vin \\
  & \eta & \mapsto & \bigotimes_{v \in V} \scL_v^{\otimes \eta(v)},
\end{array}
\end{equation}
which we extend to the map
$L_\theta \otimes \id_\bQ : \Theta_\bQ \to \Pic (\scM_\theta)_\bQ$
where the subscript denotes the tensor product with $\bQ$.
It follows from the definition that
\begin{equation} \label{eq:L_theta1}
 L_\theta(\theta) = \scO_{\scM_\theta}(1).
\end{equation}
The definition of stability condition makes sense
also for $\theta \in \Theta_\bQ = \Theta \otimes \bQ$
or $\Theta_\bR = \Theta \otimes \bR$.
For any $\theta \in \Theta_\bR$,
one can find a stability parameter $\theta' \in \Theta$
defining the same stability condition as $\theta$.
Then the moduli space $\scM_{\theta'}$ is isomorphic to $\scM_\theta$,
and the tautological bundle is the $\bR$-line bundle
defined as
$L_{\theta'}(\theta) \in \Pic(\scM_\theta)_\bR$.

%\section{Proof of Proposition \ref{pr:non-degenerate}}
% \label{sc:prop1.1}
\section{Moduli spaces associated with non-degenerate dimer models}
\label{sc:non-degenerate}

Assume that the quiver $\Gamma$ with relations
comes from a dimer model $G$.

\begin{theorem}[{\cite[Proposition 5.1]{Ishii-Ueda_08}}]
 \label{th:Ishii-Ueda_08}
If the dimer model $G$ is non-degenerate,
then the moduli space $\scM_{\theta}$ is
a smooth toric Calabi-Yau 3-fold for generic $\theta$.
\end{theorem}

The dense torus $\bT \subset \scM_\theta$ is defined as
$$
 \bT = \{ [((M_v)_{v \in V}, (\psi_a)_{a \in A})] \in \scM_\theta \mid
  \psi_a \ne 0 \text{ for any } a \in A \},
$$
which acts on $\scM_\theta$ by pointwise multiplication.
Any toric divisor is written as
$$
 \{ [((M_v)_{v \in V}, (\psi_a)_{a \in A})] \in \scM_\theta \mid
  \psi_a = 0 \text{ for any } a \in D \},
$$
for a perfect matching $D \subset E$.
%%The stabilizer group of the divisor
%%is determined by the height change of the perfect matching.
%
For a not necessarily generic stability parameter $\theta$,
the moduli space $\scMbar_\theta$ still contains $\bT$
as an open subset,
since any object in $\bT$ has no proper subobject,
and hence $\theta$-stable for any $\theta$.
We write the normalization
of the irreducible component of $\scMbar_\theta$
containing $\bT$ as $X_\theta$.
Theorem \ref{th:Ishii-Ueda_08} gives
$X_\theta = \scM_\theta$ for generic $\theta$,
and $X_0$ is the affine toric variety
associated with the cone $\bR \cdot (\Delta \times \{ 1 \}) \subset \bR^3$
over a convex lattice polygon $\Delta \subset \bR^2$
called the {\em characteristic polygon}
\cite[Proposition 6.3]{Ishii-Ueda_08}.

\section{Walls and chambers for stability parameters}
 \label{sc:walls_and_chambers}
  
A proper non-empty subset
$R \subset V$ defines a {\em wall}
$$
 W_R = \{ \theta \in \Theta_\bR \mid \theta(\chi_R) = 0 \}
$$
in the space of stability parameters, where
%$\chi_R \in \bZ^V$
$$
 \chi_R(v) =
  \begin{cases}
   1 & v \in R, \\
   0 & \text{otherwise}
  \end{cases}
$$
is the characteristic function of $R$.

\begin{lemma} \label{lm:genericity}
A stability parameter is generic
if it does not lie on any wall $W_R$.
\end{lemma}

\begin{proof}
If there is a strictly $\theta$-semistable object $E$,
then there is a subobject $S \subset E$
such that
$$
 \theta(\dim S) = 0.
$$
Since $\dim E = \chi_V$,
the dimension vector $\dim S$ is
the characteristic function
of the support $R$ of $S$,
and $\theta$ lies on the wall $W_R$.
\end{proof}

\begin{lemma} \label{lm:chamber1}
Assume $\Gamma$ comes from a non-degenerate dimer model $G$.
Let $\theta$ be a generic stability parameter.
Then the set
$$
 C = \lc \eta \in \Theta_\bR \relmid
\begin{array}{c}
 \text{For any $\bC \Gamma$-module $M$ of dimension vector $\chi_V$,} \\
 \text{$M$ is $\theta$-stable if and only if $M$ is $\eta$-stable}
\end{array}
 \rc
$$
is the convex open rational polyhedral cone
consisting of $\eta \in \Theta_\bR$ such that
$\eta(F) > 0$ for any non-trivial submodule
$F$ of a $\theta$-stable representation $M$.
\end{lemma}

\begin{proof}
Let $C'$ be the convex open cone
consisting of $\eta \in \Theta_\bR$ such that
$\eta(F) > 0$ for any non-trivial submodule
$F$ of a $\theta$-stable representation $M$.
Although the set of $\theta$-stable representations
is infinite,
the set of dimension vectors of non-trivial submodules $F$
of a $\theta$-stable representation is finite,
so that $C'$ is a rational polyhedral cone
defined by finitely many inequalities.
The inclusion $C \subset C'$ is obvious and we show the converse.
Assume for a parameter $\eta \in C'$,
a representation $E$ is $\eta$-stable but
not $\theta$-stable.
If $\eta' \in C'$ is close enough to $\eta$,
then $E$ is also $\eta'$-stable.
Replacing $\eta$ by $\eta'$, we may assume $\eta$ is generic.
Then since both $\scM_{\theta}$ and $\scM_{\eta}$ are crepant resolutions
of the same toric variety, the inclusion $\scM_{\theta} \subset \scM_{\eta}$
implies $\scM_{\theta} = \scM_{\eta}$,
contradicting $[E] \in \scM_{\eta} \setminus \scM_{\theta}$.
\end{proof}

\begin{remark}
The proof of \pref{lm:chamber1} does not show
\begin{align} \label{eq:c1}
 C =
\lc \eta \in \Theta_\bR \relmid
\begin{array}{c}
 \text{For any $\bC \Gamma$-module $M$ of dimension vector $\chi_V$,} \\
 \text{$M$ is $\theta$-semistable if and only if $M$ is $\eta$-semistable}
\end{array}
 \rc.
\end{align}
\eqref{eq:c1} is equivalent to the genericity of all the parameters in $C$,
which follows from the connectedness
of $\overline{\scM_{\eta}}$ for any $\eta \in C$.
Once we establish (\bfT)+(\bfE),
the argument of \cite[Section 8]{Bridgeland-King-Reid}
shows the connectedness of $\scMbar_\eta$ for any $\eta \in C$.
\end{remark}

The subspace $C \subset \Theta$
is called a {\em chamber},
which is separated from other chambers by walls.
The space $\Theta$ of stability parameters
is divided into finitely-many walls and chambers,
and any stability parameters in the same chamber
give the same stability condition.
Since the map
$L_\theta : \Theta \to \Pic \scM_\theta$
defined in \eqref{eq:L_theta}
depends only on the chamber $C$ containing $\theta$,
we often write $L_C$ instead of $L_\theta$.

Let $(\theta, \theta')$ be a pair of generic stability parameters
in a pair $(C, C')$ of adjacent chambers in $\Theta$
separated by a wall $W_R$.
We assume that the stability parameter
$$
 \theta_0 = \frac{1}{2}(\theta + \theta')
$$
is {\em generic on the wall}
in the sense that $\theta_0 \in W_R$ and
it does not lie on any other walls.
Recall that $X_{\theta_0}$ is the normalization
of the distinguished irreducible component of $\scMbar_{\theta_0}$
containing the dense torus $\bT \subset \scMbar_{\theta_0}$.
By the definition of GIT quotient,
there is a projective morphism
$\scM_\theta \to \scMbar_{\theta_0}$,
%and $f' : \scM_{\theta'} \to \scMbar_{\theta_0}$,
which factors through $X_{\theta_0}$
as $\scM_\theta \xto{f} X_{\theta_0} \to \scMbar_{\theta_0}$
since $\scM_\theta$ is smooth and irreducible.
The morphism
$
 f' : \scM_{\theta'} \to X_{\theta_0}
$
is defined similarly, and
one obtains the diagram
%\begin{equation} \label{eq:VGIT1} 
%\begin{psmatrix}[colsep=1cm, rowsep=1cm]
% \scM_{\theta} & & \scM_{\theta'} \\
% & \scMbar_{\theta_0}
%\end{psmatrix}
%\psset{shortput=nab,labelsep=1pt,nodesep=3pt,arrows=->}
%\ncline{1,1}{2,2}_{f}
%\ncline{1,3}{2,2}^{f'}
%\end{equation}
\begin{equation} \label{eq:VGIT} 
\begin{psmatrix}[colsep=1cm, rowsep=1cm]
 \scM_{\theta} & & \scM_{\theta'} \\
 & X_{\theta_0}
\end{psmatrix}.
\psset{shortput=nab,labelsep=1pt,nodesep=3pt,arrows=->}
\ncline{1,1}{2,2}_{f}
\ncline{1,3}{2,2}^{f'}
\end{equation}
Let $\scO_{X_{\theta_0}}(1)$ be the pull-back of
the ample line bundle $\scO_{\scMbar_{\theta_0}}(1)$
by the normalization map
$X_{\theta_0} \to \scMbar_{\theta_0}$.

\begin{lemma}
The line bundle $L_C(\theta')^{-1}$ is $f$-ample.
\end{lemma}

\begin{proof}
One has
$
 L_C(\theta_0) \cong f^* \scO_{X_{\theta_0}}(1),
$
so that
$
 L_C(\theta) \otimes L_C(\theta')
  \cong L_C(2 \theta_0)
  \cong f^* \scO_{X_{\theta_0}}(2).
$
This shows that
$
 L_C(\theta')^{-1}
  \cong L_C(\theta) \otimes f^* \scO_{X_{\theta_0}}(-2)
$
is $f$-ample.
\end{proof}

\begin{lemma}
Both $f$ and $f'$ are $\bT$-equivariant,
so that \eqref{eq:VGIT} is a diagram
in the category of toric varieties.
\end{lemma}

\begin{proof}
This is clear from the definition of $\bT$
and its action on $\scM_\theta$,
$\scM_{\theta'}$ and $X_{\theta_0}$.
\end{proof}

\section{Tilting object and derived equivalence}
 \label{sc:equivalence}

\begin{definition} \label{df:tilting}
An object $\scE$ in a triangulated category $\scT$
is {\em acyclic} if
$$
 \Ext^k(\scE, \scE) = 0, \qquad k \ne 0.
$$
It is a {\em generator} if for any object $\scF$,
$$
 \Ext^k(\scE, \scF) = 0
$$
for any $k \in \bZ$ implies $\scF \cong 0$.
An acyclic generator is called a {\em tilting object}.
\end{definition}

A tilting object induces a derived equivalence:

\begin{theorem}[Bondal \cite{Bondal_RAACS}, Rickard \cite{Rickard}]
Let $\scE$ be a tilting object
in the derived category $D^b \coh X$
of coherent sheaves on a smooth quasi-projective variety $X$.
Then the functor
\begin{equation} \label{eq:der_equiv1}
 \Phi(-)
   = \bR \Gamma(\scE \otimes - )
   : D^b \coh X \to D^b \module A
\end{equation}
is an equivalence
from $D^b \coh X$
to the derived category $D^b \module A$
of finitely-generated left modules
over the endomorphism algebra $A = \Hom(\scE, \scE)$.
\end{theorem}

It is more common to use the functor
\begin{equation} \label{eq:der_equiv2}
 \Psi(-)
   = \bR \Hom(\scE, - )
   : D^b \coh X \to D^b \module B
\end{equation}
where $B = A^\op$ is the opposite ring of $A$.
%instead of the functor $\Phi$ above.
Since $\scE$ is a tilting object
if and only if $\scE^\vee$ is a tilting object
and $B = A^\op = \End(\scE^\vee)$,
the functor $\Phi$
%\eqref{eq:der_equiv1}
is an equivalence
if and only if
the functor $\Psi$
%\eqref{eq:der_equiv2}
is an equivalence.
The advantage of the functor $\Phi$
is that
if $X = \scM_\theta$ is the moduli space of
stable representations of the quiver $\Gamma$ with relations and
$\scE = \bigoplus_v \scL_v$ is the tautological bundle,
then the functor
\begin{equation*} %\label{eq:der_equiv2}
 \Phi_\theta(-)
   = \bR \Gamma \lb \scE \otimes - \rb
   : D^b \coh \scM_\theta \to D^b \module \bC \Gamma,
\end{equation*}
sends the structure sheaf $\scO_y$
of a point $y \in \scM_\theta$
to the $\bC \Gamma$-module $E_y$
parametrized by that point;
$
 \Phi_\theta(\scO_y) = E_y.
$

The full subcategory of $\module \bC \Gamma$
consisting of finite-dimensional representations
will be denoted by $\module_c \bC \Gamma$.
A $\bC \Gamma$-module $M$ is {\em nilpotent}
if there is a natural number $n$
such that any path with length greater than $n$
acts on $M$ by zero.
The full subcategory of $\module_c \bC \Gamma$
consisting of nilpotent representations
will be denoted by $\module_0 \bC \Gamma$.
%Both $\module_c \bC \Gamma$ and
%$\module_0 \bC \Gamma$ are abelian categories.
%
The functor $\Phi_\theta$ induces functors
$
% \Phi_\theta :
  D^b \coh_c \scM_\theta \to D^b \module_c \bC \Gamma
$
and
$
% \Phi_\theta :
  D^b \coh_0 \scM_\theta \to D^b \module_0 \bC \Gamma,
$
both of which will be written as $\Phi_\theta$
by abuse of notation.
Here $\coh_c \scM_\theta$
is the full subcategory of $\coh \scM_\theta$ 
consisting of compactly supported sheaves,
and $\coh_0 \scM_\theta$ is its full subcategory
consisting of sheaves supported on the fiber
of the unique torus fixed point of $X_0$
by the natural map $\scM_\theta \to X_0$.

\section{Stability parameters and Grothendieck groups}
 \label{sc:stability_parameters}

Assume $\Gamma$ comes from a non-degenerate dimer model $G$ and
the conditions \textup{(\bfT)+(\bfE)} hold
for some generic stability parameter $\theta$.
The Grothendieck groups of
$\coh \scM_\theta$,
$\coh_c \scM_\theta$,
$\coh_0 \scM_\theta$
$\module \bC \Gamma$,
$\module_c \bC \Gamma$ and
$\module_0 \bC \Gamma$
will be denoted by
$K(\scM_\theta)$,
$K_c(\scM_\theta)$,
$K_0(\scM_\theta)$,
$K(\Gamma)$,
$K_c(\Gamma)$ and
$K_0(\Gamma)$ respectively.
Since no path of length zero appears in the relations of $\Gamma$,
the Grothendieck group $K_0(\Gamma)$ is the abelian group
freely generated by the classes of simple modules $S_v$
supported on vertices $v$ of the quiver.
Our assumption implies that
the path algebra $\bC \Gamma$ has finite global dimension,
so that $K(\Gamma)$ is isomorphic to its subgroup
$K^\perf(\Gamma)$ generated by the classes of projective modules.
For each vertex $v \in V$, there is a projective module
$P_v = \bC \Gamma \cdot e_v$,
where $e_v \in \bC \Gamma$ is the idempotent associated with $v$.
There is a natural pairing
\begin{equation} \label{eq:Euler_form1}
 \chi : K^\perf(\Gamma) \times K_c(\Gamma) \to \bZ, \qquad
  (\alpha, \beta) \mapsto \sum_i (-1)^i \dim \Ext^i(\alpha, \beta)
\end{equation}
called the {\em Euler form},
and the {\em numerical Grothendieck group} $N_c(\Gamma)$
is defined as the quotient of $K_c(\Gamma)$
by the radical of the Euler form.
It is clear that
$$
 \Ext^i(P_v, S_w) =
  \begin{cases}
   \bC & v = w \text{ and } i = 0, \\
   0 & \text{otherwise},
  \end{cases}
$$
so that one has the following:

\begin{lemma} \label{lm:dual_bases}
Assume that
\begin{itemize}
 \item
$K(\Gamma)$ and $K_0(\Gamma)$ have the same rank,
 \item
$K(\Gamma)$ is torsion-free, and
 \item
$K_0(\Gamma) \to N_c(\Gamma)$ is surjective,
\end{itemize}
then the map $K_0(\Gamma) \to N_c(\Gamma)$ is an isomorphism,
and $([\scP_v])_{v \in V}$ and $([S_v])_{v \in V}$ are dual bases
of $K(\Gamma)$ and $N_c(\Gamma)$.
\end{lemma}

%Assume $\Gamma$ comes from a non-degenerate dimer model $G$ and
%the conditions \textup{(\bfT)+(\bfE)} hold
%for some generic stability parameter $\theta$.
%Then $\bC \Gamma$ is of finite global dimension.
\begin{lemma}
The rank of $K(\Gamma) \cong K(\scM_\theta)$ coincides
with the Euler number of $\scM_\theta$,
which coincides with the rank of
$
 K_0(\Gamma)
  \cong K_0(\scM_\theta).
$
\end{lemma}
\begin{proof}
Since the rank of the Grothendieck group coincides with that of
the Chow group, we may consider Chow groups instead of Grothendieck groups.
Then the assertion follows
from the argument of \cite[Section 5.2]{Fulton_ITV},
which is sketched below for the readers' convenience.

Let $\Sigma$ be the fan which determines the toric variety $\scM_\theta$
and let $e$ be the number of $3$-dimensional cones of $\Sigma$.
We can construct a complete simplicial fan $\Sigmatilde$ containing $\Sigma$
such that $\Sigmatilde$ has one extra one-dimensional cone and
$n$ extra three-dimensional cones,
where $n$ is the number of edges of the characteristic polygon.
Then $\Sigmatilde$ determines a projective toric variety $X(\Sigmatilde)$.

If we take the ordering of the three-dimensional cones of $\Sigmatilde$
in \cite[Section 5.2]{Fulton_ITV}
so that the last $e$ cones are those of $\Sigma$,
then in the filtration
$X(\Sigmatilde)=Z_1 \supset Z_2 \supset \dots \supset Z_{n+e}$
by closed subschemes in \cite[Section 5.2]{Fulton_ITV},
the subscheme $Z_{n+1}$ is the fiber of the unique torus fixed point of $X_0$
by the resolution $\scM_\theta$ and $Z_i \setminus Z_{i+1}$ is an affine space
for $i \ge n+1$.
Hence $e$ coincides with the rank of the Chow group of $Z_n$,
which is equal to the rank of $K_0(\scM_\theta)$.

If we take the ordering so that the first $e$ cones are those of $\Sigma$,
then $Z_1 \setminus Z_{e+1}$ coincides with $\scM_\theta$
and $Z_i \setminus Z_{i+1}$ is an affine space for $i \le e$.
Then $e$ also coincides with the Euler number of $\scM_\theta$ and
the rank of the Chow group of $\scM_\theta$.
\end{proof}

The group $K(\Gamma) \cong K(\scM_\theta)$ is torsion-free
since $\scM_\theta$ is a toric variety.
There is a $\bC^*$-action on the representations of $\Gamma$ as in
\cite[Section 5]{Ishii-Ueda_08}
and the limit of an object in $\module_c \bC \Gamma$ with respect to this action
lies in $\module_0 \bC \Gamma$.
It follows that the map $K_0(\Gamma) \to N_c(\Gamma)$ is surjective.
This concludes the proof of the following:

\begin{proposition} \label{pr:perfect}
The map
$
 K_0(\Gamma) \simto N_c(\Gamma)
$
is an isomorphism,
the Euler form \eqref{eq:Euler_form1} is perfect, and
$([\scP_v])_{v \in V}$ and $([S_v])_{v \in V}$
are dual bases
of $K(\Gamma)$ and $N_c(\Gamma)$.
\end{proposition}

The space $\bZ^V$ of dimension vectors can naturally be identified
with the numerical Grothendieck group $N_c(\Gamma)$.
The dual space $\Hom(\bZ^V, \bZ)$ is identified with
the Grothendieck group $K(\Gamma)$
by the Euler form,
and the space
$$
 \Theta = \{ \theta \in \Hom(\bZ^V, \bZ) \mid \theta(\chi_V) = 0 \}
$$
of stability parameters is identified with its codimension one subspace.

\section{The Grothendieck group and the Picard group}
 \label{sc:KvsPic}

Let $\theta \in \Theta$ be a generic stability parameter.
Assume that the conditions (\bfT)+(\bfE) hold
for the stability parameter $\theta$,
so that one has a derived equivalence
\begin{equation} \label{eq:der_equiv3}
 \Phi_C(-)
   = \bR \Gamma \lb \scE \otimes - \rb
   : D^b \coh \scM_\theta \simto D^b \module \bC \Gamma,
\end{equation}
where we have added the subscript $C$
to indicate the dependence
on the chamber $C \subset \Theta$ containing $\theta$.
%Let $D^b \coh_0 \scM_\theta$
%be the full subcategory of $D^b \coh \scM_\theta$
%consisting of complexes
%whose cohomology sheaves are supported
%on the fiber of the unique torus fixed point of $X_0$
%by the resolution $\scM_\theta \to X_0$,
%and $D^b \module_0 \bC \Gamma$
%be the full subcategory
%of $D^b \module \bC \Gamma$
%consisting of nilpotent representations.
The functor $\Phi_C$ induces a functor
\begin{equation} \label{eq:der_equiv4}
 \Phi_C
%   = \bR \Gamma \lb \lb \bigoplus_v \scL_v \rb \otimes - \rb
   : D^b \coh_0 \scM_\theta \simto D^b \module_0 \bC \Gamma,
\end{equation}
which is denoted by the same symbol
by abuse of notation.

Let $K(\scM_\theta)$ and $K_0(\scM_\theta)$
be the Grothendieck groups
of $D^b \coh \scM_\theta$ and $D^b \coh_0 \scM_\theta$
respectively.
One has a decreasing filtration
$$
 K(\scM_\theta)_\bQ = F^0
  \supset F^1 \supset F^2 \supset F^3 \supset F^4 = 0
$$
of $K(\scM_\theta)_\bQ = K(\scM_\theta) \otimes \bQ$
by codimension of the support, and
an increasing filtration
$$
 0 = F_{-1} \subset F_0 \subset F_1
  \subset F_2  = K_0(\scM_\theta)_\bQ
$$
of $K_0(\scM_\theta)_\bQ = K_0(\scM_\theta) \otimes \bQ$
by dimension of the support.
Since the Euler form
\begin{equation} \label{eq:Euler}
 \chi : K(\scM_\theta) \times K_0(\scM_\theta) \to \bZ, \quad
 (\alpha, \beta)
  \mapsto \sum_i (-1)^i \dim \Ext^i(\alpha, \beta)
\end{equation}
is perfect by \pref{pr:perfect},
one has
$$
 F^1 = F_0^\bot, \quad
 F^2 = F_1^\bot, \quad \text{and} \quad
 F^3 = F_2^\bot = 0
$$
with respect to the Euler form
\cite[Proposition 5.1]{Craw-Ishii}.
The equivalence \eqref{eq:der_equiv4} induces an isomorphism
$
 \varphi_C : K_0(\scM_\theta) \to K_0(\Gamma),
$
whose adjoint with respect to the pairing \eqref{eq:Euler}
gives an isomorphism
\begin{equation} \label{eq:varphi_C1}
 \varphi_C^* : K(\Gamma) \to K(\scM_\theta).
\end{equation}
The isomorphism \eqref{eq:varphi_C1} induces
an isomorphism
\begin{equation} \label{eq:varphi_C2}
 \varphi_C^* : \Theta \to F^1
\end{equation}
since
$
 \Theta
  \subset \Hom(\bZ^V, \bQ)
  \cong \Hom(K_0(\Gamma), \bQ)
  \cong K(\Gamma) \otimes \bQ
$
is the kernel of $\chi_V \in \bZ^V \cong K_0(\Gamma)$
and $\varphi_C^*(\chi_V)$ is
the class of the structure sheaf of a point.
Let
$$
 \det^{-1} : F^1 / F^2 \simto \Pic(\scM_\theta)_\bQ
$$
be the isomorphism
sending the class $[\scO_D] = [\scO] - [\scO(-D)]$
of the structure sheaf of a divisor $D$
to $\det(\scO_D)^{-1} = \scO(-D)$.
Note that one has
$$
 \varphi_C^*(\theta)
  = \varphi_C^* \lb \sum_{v \in V} \theta(v) [\scP_v] \rb
  = \sum_{v \in V} \theta(v) [\scL_v^{-1}],
$$
so that the map $L_C$ defined in \eqref{eq:L_theta} can be expressed as
$$
 L_C(\theta) = \det^{-1}([\varphi_C^*(\theta)])
  \in \Pic(\scM_\theta)_\bQ.
$$
This gives the commutative diagram
$$
\begin{CD}
 \Theta @>{\varphi_C^*}>> F^1 \\
  @V{L_C}VV @VV{p}V \\
 \Pic(Y)_\bQ @<{\sim}<< F^1/F^2.
\end{CD}
$$
In particular,
one has
\begin{equation} \label{eq:degree}
 \deg(L_C(\theta)|_\ell)
  = \sum_{v \in V} \chi(\scL_v \otimes \scO_\ell) \theta(v)
  = \theta(\varphi_C(\scO_\ell))
\end{equation}
for a curve $\ell$ on $\scM_\theta$.

Recall that
a {\em contraction} is a birational morphism
onto a normal variety
with a smaller Picard number.
A contraction is {\em primitive}
if it cannot be obtained
as the composition of two contractions.
According to Wilson
\cite{Wilson_KCCYT},
the codimension one faces of the ample cone
$\Amp(Y)$
of a Calabi-Yau 3-fold
correspond to primitive birational contractions,
which can be classified into the following three types:
\begin{description}
\item[type I :] $f$ contracts a curve to a point.
\item[type \II :] $f$ contracts a surface to a point.
\item[type \III :] $f$ contracts a surface to a curve.
\end{description}

\begin{lemma}
The morphism $f : \scM_\theta \to X_{\theta_0}$
is either a primitive contraction or an isomorphism.
\end{lemma}

\begin{proof}
It follows from the definitions of $X_{\theta_0}$
and the morphism $\scM_\theta \to X_{\theta_0}$
that $L_C(\theta_0) = f^* \scO_{X_{\theta_0}}(1)$.
If $L_C(\theta_0)$ is
in the interior of the ample cone,
then $f$ is an isomorphism.
If $L_C(\theta_0)$ is on the boundary of the ample cone,
then it is in the interior of a facet
since $\theta_0$ is assumed to be outside
of any other walls and
the map $L_C$ is submersive.
It follows that $f$ is a primitive contraction
in this case.
\end{proof}

The {\em type} of a wall in the parameter space
is defined as the type
of the corresponding primitive birational contraction.
If the morphism $f : \scM_\theta \to X_{\theta_0}$
is an isomorphism,
then it is called a wall of {\em type 0}.

\section{Variations of moduli spaces}
 \label{sc:variation_of_moduli}

Let $W_{R_1}$ be the wall separating $\theta$ and $\theta'$,
where $R_1 \subset V$ is the support
of a $\theta'$-destabilizing submodule $S$
of a $\theta$-stable $\bC\Gamma$-module $E$.
Then $R_2:=V \setminus R_1$ is
the support of the quotient module $Q=E/S$
and one has $W_{R_1} = W_{R_2}$.

The {\em unstable locus} is the set
\begin{equation} \label{eq:Z}
Z=\{ y \in \scM_{\theta} \mid \text{$E_y$ is not $\theta_0$-stable}\}
 \subset \scM_{\theta}
\end{equation}
of strictly $\theta_0$-semistable representations,
with the natural subscheme structure
defined by setting the values of the arrows
going out from $R_1$ to zero.
Since the values of the arrows in the universal
representation are given by square free monomials in local coordinates
given in \cite[Lemma 4.5]{Ishii-Ueda_08},
the subscheme $Z$ is reduced.

For a point $z \in Z$,
the $\theta'$-destabilizing sequence
of the $\theta$-stable representation $E_z$
parametrized by $z$ is written as
$$
0 \to S_z \to E_z \to Q_z \to 0,
$$
where $\dim S_z = \chi_{R_1}$, $\dim Q_z = \chi_{R_2}$ and
$\theta(\chi_{R_1}) = - \theta(\chi_{R_2}) > 0$,
$\theta'(\chi_{R_1}) = - \theta'(\chi_{R_2}) < 0$.

Since $\theta_0$ is assumed to be generic on the wall,
both $S$ and $Q$ are $\theta_0$-stable.
%(and hence $\theta$-stable and $\theta'$-stable)
The point $f(z) \in \scM_{\theta_0}$ parametrizes
the S-equivalence class of $[S_z \oplus Q_z]$.

\begin{lemma}%[{cf.~\cite[Lemma 3.7]{Craw-Ishii}}]
 \label{lm:fiber}
The fibers of the morphisms $f$ and $f'$
in \eqref{eq:VGIT}
at $f(z) \in Z_0$ are
$\bP(\Ext^1(Q_z, S_z)^\vee)$ and
%the fiber of $f'$ at $[S \oplus Q] \in Z_0$ is 
$\bP(\Ext^1(S_z, Q_z)^\vee)$
respectively.
\end{lemma}

\begin{proof}
Any $\theta$-stable representation of $\Gamma$
in the same S-equivalence class as $[S_z \oplus Q_z]$
is obtained as a non-trivial extension of $Q_z$ by $S_z$, and
$\theta$-stability of $S_z$ and $Q_z$ implies that
any such extension is $\theta$-stable.
Similarly, 
$\theta'$-stable representation of $\Gamma$
in the same S-equivalence class as $[S_z \oplus Q_z]$
is obtained as a non-trivial extension of $S_z$ by $Q_z$, and
$\theta'$-stability of $S_z$ and $Q_z$ implies that
any such extension is $\theta'$-stable.
\end{proof}

%The projectivity of $f$ and $f'$
%%\ref{it:projective}
%comes
%from the construction
%of the moduli spaces
%by King \cite{King}.
%%\ref{it:surjectivity}
%The unique 3-dimensional irreducible component
%of $\scMbar_{\theta_0}$
%is obtained as the closure of $\bT \subset \scMbar_{\theta_0}$,
%and other irreducible components,
%if any, have dimensions less than or equal to two.
%The surjectivity of $f$ and $f'$
%to this component is proved
%in \cite{Craw-Ishii}.

Recall from Section \ref{sc:definitions} that
the set $V$ of vertices
is the set of connected components
of the complement $T \setminus (\bigcup_{e \in E} e)$
of the union of edges of the dimer model.
We identify $R_1$ and $R_2$
with the closures of $\bigcup_{v \in R_1} v$
and $\bigcup_{v \in R_2} v$ in $T$ respectively.

\begin{lemma} \label{lm:R_connected}
$R_1$ and $R_2$ are connected.
\end{lemma}

\begin{proof}
Suppose $R_1=R_{1,1} \cup R_{1,2}$ is a disjoint union
of two open subsets of $R_1$.
Then $S_z$ is a direct sum of two submodules
whose dimension vectors are
$\chi_{R_{1,1}}$ and $\chi_{R_{1,2}}$ respectively.
%Since $E_z$ is $\eta$-semistable for any $\eta \in W_{R_1}$,
%we have $\eta(\chi_{R_{1,1}})=\eta(\chi_{R_{1,2}})=0$.
If both $R_{1,1}$ and $R_{1,2}$ are non-empty,
then $S_z$ is the direct sum of two non-zero submodules,
which cannot be $\theta_0$-stable.
%then $\eta(\chi_{R_{1,1}})$ and $\eta(\chi_{R_{1,2}})$ are
%independent linear function on $\eta$,
%which contradicts $\codim W_{R_1}=1$.
Hence $R_1$ is connected.
The connectedness of $R_2$ is proved similarly.
\end{proof}

\begin{lemma} \label{lm:simple}
$S_z$ and $Q_z$ are simple;
$\End(S_z) = \bC \cdot \id_{S_z}$ and
$\End(Q_z) = \bC \cdot \id_{Q_z}$.
\end{lemma}

\begin{proof}
This follows from the $\theta_0$-stability.
%Since the dimension vector of $S_z$ is {\em thin}
%in the sense that $(\dim S_z)(v)$ is $0$ or $1$ for any $v \in V$,
%the existence of a non-identity endomorphism
%implies that $S_z$ decomposes into a direct sum
%of proper subobjects.
%Since $\theta_0$ is generic on the wall,
%one of them destabilizes $S_z$ and
%contradicts the $\theta_0$-stability of $S_z$.
%The same argument also shows
%the simplicity of $Q_z$.
\end{proof}

%The following is a generalization
%of a claim in the proof of
%\cite[Proposition 10.6]{Craw-Ishii}:

\begin{lemma} \label{lm:dim_Ext}
If $z \in Z$ is in a two-dimensional $\bT$-orbit,
then $\dim \Ext^1(Q_z, S_z)$ is the number of connected components
of the boundary $\partial R_1$ of $R_1$.
\end{lemma}

\begin{proof}
According to \cite[Lemma 6.1]{Ishii-Ueda_08},
a point $z \in Z$ is in a two-dimensional $\bT$-orbit
if and only if the arrows whose values are zero
in the representation $E_z$ form a perfect matching.
By using the $\bT$-action, we may assume that
all the non-zero values of arrows are $1$.
Since $S_z$ is a submodule of $E_z$,
one must have $\psi_a = 0$
for any arrow $a$ from $R_1 = \Supp S_z$
to $R_2 = \Supp Q_z = V \setminus R_1$.
This means that $R_1$ is a ``sink''
in the sense that any arrows going from $R_1$
is blocked by the perfect matching.

To give an extension
\begin{equation} \label{eq:extension}
 0 \to S_z \to E \to Q_z \to 0
\end{equation}
of $Q_z$ by $S_z$ is the same thing
as to assign numbers $\psi_a \in \bC$ to each arrows $a$
from $R_2 = \Supp Q$ to $R_1 = \Supp S$.
If two arrows $a$ and $b$ from $R_2$ to $R_1$
cross the same connected component of $\partial R_1$,
then one must have $\psi_a = \psi_b$
by the relations of the quiver.
If they cross different components, then the values are independent.
Since $\Ext^1(Q_z, S_z)$ classifies extensions
of the form \eqref{eq:extension},
Lemma \ref{lm:dim_Ext} is proved.
\end{proof}

\begin{lemma}%[{cf.~\cite[Proposition 3.8]{Craw-Ishii}}]
 \label{lm:abscence_II}
There are no walls of type \II.
\end{lemma}

\begin{proof}
Assume that
$f : \scM_\theta \to X_{\theta_0}$ contracts
a surface to a point,
and take a point $z$
in a two-dimensional $\bT$-orbit
on the contracted surface.
Then for the $\theta'$-destabilizing sequence
$$
 0 \to S_z \to E_z \to Q_z \to 0,
$$
one has
$\dim \Ext^1(Q_z, S_z) = 3$
by Lemma \ref{lm:fiber},
so that the the boundary of $R_1$ has
three connected components
by Lemma \ref{lm:dim_Ext}.
Lemma \ref{lm:R_connected}
shows that both $R_1$ and $R_2$ are connected.
On the other hand,
it is impossible to divide the torus $T$
into two connected pieces $R_1$ and $R_2$
in such a way that the boundary of $R_1$ has
three connected components,
and Lemma \ref{lm:abscence_II} is proved.
\end{proof}

\section{Variations of tautological bundles}
 \label{sc:variation_of_bundles}

Assume that the conditions (\bfT)+(\bfE) hold for $\theta$, and
consider the diagram \eqref{eq:VGIT}.
The resulting birational morphism
$\scM_\theta \dashrightarrow \scM_{\theta'}$
is either a flop or %in the type I case and
an isomorphism %in the type 0 or type \III\  case.
since both are crepant resolutions of $X_{\theta_0}$
by \cite[Proposition 5.1]{Ishii-Ueda_08}.
In either cases,
the strict transformation
induces an isomorphism
$$
 \Pic \scM_\theta \simto \Pic \scM_{\theta'}
$$
of Picard groups.
Let $\scL_v$ and $\scL_v'$ be the tautological bundles
on $\scM_\theta$ and $\scM_{\theta'}$ respectively.

We prove the following in this section:

\begin{theorem} \label{th:Xi}
There is an equivalence
$$
 \Xi : D^b \coh \scM_\theta \simto D^b \coh \scM_\theta
$$
of the derived category such that
$\Xi(\scL_v) = \scL_v'$
for any $v \in V$.
\end{theorem}

This result is proved
in \cite{Craw-Ishii}
for McKay quivers
by assuming the conditions (\bfT)+(\bfE)
for both $\theta$ and $\theta'$.
%\cite[Proposition 5.5(ii) and Corollary 5.6]{Craw-Ishii}. 
Here
we work with general dimer models,
and we assume the conditions (\bfT)+(\bfE)
only for $\theta$.
Since there are only finitely many walls,
\pref{th:change_of_stability} is an immediate consequence
of \pref{th:Xi}.

\subsection{Type 0 case}

Assume that
$
 f : \scM_\theta \to X_{\theta_0}
$
is an isomorphism,
so that
the diagram \eqref{eq:VGIT} induces
an isomorphism
$
 \phi : \scM_\theta \simto \scM_{\theta'}
$
of moduli spaces,
and only tautological bundles change.

Recall from \eqref{eq:Z}
that the unstable locus $Z$ is defined
by setting the arrows going out from $R_1$ to zero.
The zero locus of each arrow is a union of reduced toric divisors,
and $Z$ is the intersection of such divisors.
\begin{lemma}
Every irreducible component of $Z$ is two-dimensional.
\end{lemma}
\begin{proof}
Recall that a family of representations of a quiver is given by
a collection of line bundles and maps of the line bundles.
Therefore, if two families of representations parametrized by
a normal variety coincide in codimension one,
then they must be the same family.
Since $\scM_\theta \setminus Z = \scM_{\theta'}\setminus Z$ parametrizes
the same family and $Z$ is the locus where the two families are different,
we are done.
\end{proof}

Let $Z_1$ be a connected component of $Z$.
%For a closed point $y \in \scM_{\theta}$,
%we write the corresponding $\theta$-stable representation
%as $E_y$.
Take $z \in Z_1$
and let
$$
 0 \to S_z \to E_z \to Q_z \to 0
$$
be $\theta'$-destabilizing sequence of $E_z$.

\begin{lemma}%[{cf.~\cite[Corollary 4.6]{Craw-Ishii}}]
Either $S_z$ or $Q_z$ is rigid for any $z \in Z_1$.
\end{lemma}

\begin{proof}
Since $f$ is of type 0,
one has $\dim \Ext^1(Q_z, S_z) = 1$ by Lemma \ref{lm:fiber}.
Assume $z$ is in a two-dimensional $\bT$-orbit.
Then Lemma \ref{lm:dim_Ext} shows that
$\partial R_1$ is connected.
This implies that
either $R_1$ or $R_2$ is simply connected.
If $R_1$ is simply-connected,
then the values $\psi_a$ for the representation $S_z$
can be normalized to $0$ or $1$
by gauge transformation,
so that $S_z$ is rigid.
If not,
then $R_2$ is simply-connected and
$Q_z$ is rigid.

Since $S_z$ and $Q_z$ are $\theta_0$-stable for any $z \in Z_1$,
we have morphisms from $Z_1$ to the moduli spaces of $\theta_0$-stable
representations with dimension vectors $\chi_{R_1}$ and $\chi_{R_2}$ respectively.
A rigid representation forms a connected component of the moduli space
and therefore if a rigid representation appears in a flat family of $\theta_0$-stable
representations parametrized by the connected scheme $Z_1$, then it must be a constant family.
\end{proof}

\subsubsection{Rigid subobject case}
 \label{sc:rigid_subobject}

We first assume that $S_z$ is rigid.
This assumption implies that
the family $\scS$ of destabilizing submodules
of the universal representation over $Z_1$
is constant.
Let $S$ be the rigid destabilizing submodule
of $E_z$ for some (and hence all) $z \in Z_1$.

\begin{lemma} \label{lm:support_compact}
$\Supp(\Phi^{-1}(S)) \subset \scM_\theta$ is compact.
\end{lemma}

\begin{proof}
Note that $\Supp(\Phi^{-1}(S))$ is projective over
the affine variety $X_\theta$.
Therefore, if $\Supp(\Phi^{-1}(S))$ is not compact,
then $\End(\Phi^{-1}(S))$ must be infinite-dimensional,
which contradicts
$
 \End(\Phi^{-1}(S))
  \cong \End(S)
  = \bC \cdot \id_S
$
in Lemma \ref{lm:simple}.
%If $\Supp(\Phi^{-1}(S))$ is the whole of $\scM_\theta$,
%then $\End(\Phi^{-1}(S))$ must be infinite-dimensional,
%which contradicts
%$
% \End(\Phi^{-1}(S))
%  \cong \End(S)
%  = \bC \cdot \id_S
%$
%in Lemma \ref{lm:simple}.
\end{proof}

\begin{lemma} \label{lm:chiSE}
$\chi(S, E_y) = 0$ for any $y \in \scM_\theta$.
\end{lemma}

\begin{proof}
Recall that $\chi(S, E_y) = \sum_{i=0}^3 (-1)^i \dim \Ext^i(S, E_y)$
is the Euler characteristic of $\bR \Hom(S, E_y)$.
Since $\Supp(\Phi^{-1}(S))$ is compact,
one can take $y \in \scM_\theta \setminus \Supp(\Phi^{-1}(S))$,
which clearly satisfies $\chi(\Phi^{-1}(S), \scO_y) = 0$.
Since skyscraper sheaves are numerically equivalent,
%(i.e. they give the same element
%in the numerical Grothendieck group),
this shows $\chi(S, E_y) = \chi(\Phi^{-1}(S), \scO_y) = 0$
for any $y \in \scM_\theta$.
\end{proof}

\begin{lemma} \label{lm:ExtSE}
For any $y \in \scM_\theta \setminus Z$
and any $i \in \bZ$,
one has $\Ext^i(S, E_y) = 0$.
\end{lemma}

\begin{proof}
Since both $S$ and $E_y$ are $\bC \Gamma$-modules,
the Calabi-Yau property of
$D^b \coh \scM_\theta \cong D^b \module \bC \Gamma$
implies $\Ext^i(S, E_y) = 0$ for $i \ne 0, 1, 2, 3$.
For $i = 0$, one has
$\Hom(S, E_y) = 0$ if $y \nin Z$,
since the image of $S$ in $E_y$ will be
a $\theta'$-destabilizing submodule
otherwise.
For $i = 3$,
one has
$
 \Ext^3(S, E_y)
%  \cong \Ext^3(\Phi^{-1}(S), \scO_y)
%  \cong \Hom(\scO_y, \Phi^{-1}(S))^\vee
  \cong \Hom(E_y, S)^\vee
  = 0,
$
since the kernel of an element in
$\Hom(E_y, S)$ will be
a $\theta$-destabilizing submodule of $E_y$.
For $i=2$ and a point $z\in Z_1$,
consider the Jordan-H{\"o}lder filtration
$$
 0 \to S \to E_z \to Q_z \to 0
$$
of $E_z$.
The assumption $y \nin Z$ implies $\Hom(E_y, Q_z)=0$.
Since $\Phi$ is an equivalence,
we have $\Ext^1(E_y, E_z) \cong \Ext^1(\scO_y, \scO_z) = 0$,
which implies $\Ext^1(E_y, S)=0$.
The remaining case $i = 1$
follows from
$\chi(S, E_y) = 0$
shown in Lemma \ref{lm:chiSE}.
\end{proof}

\begin{lemma} \label{lm:support}
One has $\Supp \Phi^{-1}(S) = Z_1$.
\end{lemma}

\begin{proof}
One has
$\Hom^i(\Phi^{-1}(S), \scO_y) = 0$
for any $y \nin Z$ and any $i \in \bZ$,
which implies
$\Supp \Phi^{-1}(S) \subset Z$.
On the other hand,
one also has $\Hom(S, E_y) \cong \bC$ if $y \in Z_1$,
so that $\Supp \Phi^{-1}(S) \supset Z_1$.
Indecomposability of $S$ implies that of $\Phi^{-1}(S)$,
so that the support of $\Phi^{-1}(S)$ is connected.
Since $Z_1$ is a connected component,
one obtains $\Supp \Phi^{-1}(S) = Z_1$.
\end{proof}

\begin{lemma} \label{lm:Z1_compact}
$Z_1$ is compact.
\end{lemma}

\begin{proof}
This is clear from Lemma \ref{lm:support_compact}
and Lemma \ref{lm:support}.
%Since $Z_1$ is an intersection of toric divisors,
%the same reasoning as Lemma \ref{lm:support_compact}
%shows that $Z_1$ is compact.
\end{proof}

\begin{lemma}%[{cf. \cite[Proposition 4.4]{Craw-Ishii}}]
 \label{lm:connected}
$Z$ is connected.
\end{lemma}

\begin{proof}%[Proof of Lemma \ref{lm:connected}]
Assume for contradiction that
there is another connected component $Z_2$ of $Z$.
First we consider the case
when $Z_2$ parametrizes representations
whose destabilizing subobject $S'$ is rigid.
The classes $[S]$ and $[S']$ in the Grothendieck group
$K(\Gamma)$ are the same,
since both of them are in the kernel
of a stability parameter $\theta_0$
which is generic on the wall.
This contradicts the fact that
the classes of $\Phi^{-1}(S)$ and $\Phi^{-1}(S')$
in $K(\scM_\theta)$ must be linearly independent,
since their supports are unions of compact toric divisors
which are mutually disjoint.

Next we consider the case
when $Z_2$ parametrizes representations
containing a rigid quotient $Q$.
The support of $\Phi^{-1}(Q)$ coincides with $Z_2$
as in the subrepresentation case.
On the other hand,
the sum $[S]+[Q]$ of the classes of $S$ and $Q$ in $K(\Gamma)$ is zero.
This contradicts the linear independence
of $[\Phi^{-1}(S)]$ and $[\Phi^{-1}(Q)]$ in $K(\scM_\theta)$,
obtained in just the same way as above.
\end{proof}

\begin{lemma}%[{cf. \cite[Proposition 4.4]{Craw-Ishii}}]
 \label{lm:compact}
$Z$ is compact.
\end{lemma}

\begin{proof}
This is an immediate consequence
of Lemma  \ref{lm:Z1_compact} and Lemma \ref{lm:connected}.
\end{proof}

\begin{lemma} \label{lm:ExtSE2}
For any $z \in Z$,
one has
$$
 \Hom^i(\Phi^{-1}(S), \scO_z)
  = \begin{cases}
       \bC & i = 0, 1, \\
       0 & \text{otherwise}.
     \end{cases}
$$
\end{lemma}

\begin{proof}
For $i = 0$, one has
$\Hom(S, E_z) \cong \bC$
for $z \in Z$.
For $i = 3$,
one has
$\Hom(E_z, S) = 0$ for any $z \in \scM_\theta$
by $\theta$-stability of $E_z$.
For $i = 2$,
consider the $\theta'$-destabilizing sequence
\begin{equation} \label{eq:destab}
 0 \to S \to E_z \to Q_z \to 0.
\end{equation}
One has $\dim \Ext^1(Q_z, S)=1$
by Lemma \ref{lm:fiber}, and
rigidity of $S$ implies $\dim \Ext^1(S, S) = 0$.
Now the long exact sequence of $\Ext^*(-, S)$
associated with \eqref{eq:destab}
shows $\Ext^1(E_z, S) = 0$.
The remaining case $i = 1$
follows from $\chi(\Phi^{-1}(S), \scO_z) = 0$
just as in Lemma \ref{lm:ExtSE}.
\end{proof}

\begin{lemma} \label{lm:Cartier1}
$\Phi^{-1}(S)$ is a line bundle
on a Cartier divisor
whose reduced part is $Z$.
\end{lemma}

\begin{proof}
Lemma \ref{lm:ExtSE2} shows that
the cohomology sheaves
$\scH^i(\Phi^{-1}(S))$ is non-zero
only if $i = 0$ or $-1$,
and both of them are supported on $Z$.
Then by \cite[Proposition 5.4]{MR1910263},
the object $\Phi^{-1}(S)$ has homological dimension at most one,
i.e., the object $\Phi^{-1}(S)$ is quasi-isomorphic to a complex of the form
$$
  0
  \to \scP^{-1}
  \to \scP^0
  \to 0
$$
where $\scP^i$ are locally-free sheaves.
Since the support of $\Phi^{-1}(S)$ is a proper subset,
the locally-free sheaves $\scP^{-1}$ and $\scP^0$ have the same rank
and $\Phi^{-1}(S)$ is isomorphic to an $\scO_{\scM_\theta}$-module.
%(i.e., a complex concentrated in degree zero).
The condition $\Hom(\Phi^{-1}(S), \scO_z) \cong \bC$ for $z \in Z$ implies
that locally near $z$, one can take a surjection
$\scQ^0 \to \Phi^{-1}(S)$ from a line bundle $\scQ^0$.
The kernel $\scQ^{-1}$ of this surjection is a line bundle
since $\Phi^{-1}(S)$ has homological dimension one.
It follows that $\Phi^{-1}(S)$ is locally isomorphic
to the cokernel of a morphism of line bundles, and
\pref{lm:Cartier1} is proved.
\end{proof}

%\begin{corollary}
%$Z$ is a Cartier divisor.
%\end{corollary}

\begin{lemma}
$\Phi^{-1}(S)$ is a line bundle on $Z$.
\end{lemma}

\begin{proof}
Lemma \ref{lm:Cartier1} shows that
$\Psi^{-1}(S)$ is a sheaf of $\scO_{n Z}$-module
for sufficiently large $n$.
On the other hand,
the sum $\omega$ of all the small cycles
lies in the center of $\bC \Gamma$ and acts on $S$ as zero.
The zero locus of $\omega$ as a function on $\scM_{\theta}$
is the union of all the toric divisors with multiplicity one.
This shows that $\Phi^{-1}(S)$ in fact is a sheaf of $\scO_Z$-modules.
\end{proof}

\begin{lemma}
For any vertex $v \in R_1$,
one has an isomorphism
$$
 \Phi^{-1}(S) \simto \scL_v^\vee|_Z
$$
of line bundles on $Z$.
\end{lemma}

\begin{proof}
%Change the normalization
%of the tautological line bundle
%to $\scL_v^\vee \otimes \scE$,
%so that the tautological line bundle
%$
% \scL_v^\vee \otimes \scL_v
%  \subset \scL_v^\vee \otimes \scE
%$
%associated with $v$ is the trivial bundle.
%With respect to this new tautological bundle,
%the universal representation at $z \in Z$ is given by
%$
% E_z = ( \scL_v^\vee \otimes \scE )_z,
%$
%and the $\theta'$-destabilizing subobject is given by
%$
% S = ( \bigoplus_{w \in R_1} \scL_v^\vee \otimes \scL_w )_z.
%$
Rigidity of $S$ implies that
any non-zero arrow between two vertices in $R_1$
does not vanish on $Z$,
so that the restrictions $\scL_w|_Z$
% $(\scL_v^\vee \otimes \scL_w) |_Z$
for $w \in R_1$ are mutually isomorphic. 
It follows that $(\scL_v^\vee \otimes \scL_w) |_Z \cong \scO_Z$
for any $w \in R_1$, and one has
$
 S = H^0 \lb \bigoplus_{w \in R_1}
       \lb \scL_w \otimes \scL_v ^\vee \rb |_Z \rb.
$
Note that
$
 \bigoplus_{w \in V} \scL_w \otimes \scL_v^\vee
$
is the universal family,
which is normalized in such away that
the line bundle corresponding to the vertex $v$ is trivial,
and $S \otimes_\bC \scO_Z$ is the (trivial) family
of destabilizing subobjects
of the restriction of this universal family to $Z$.

It follows from the definition of the functor $\Phi$ that
\begin{equation} \label{eq:S}
 S
  = H^0 \lb \lb \textstyle{\bigoplus_{w \in R_1}}
      \scL_w \rb \otimes \scL_v^\vee|_Z \rb
  \subset H^0 \lb \lb \textstyle{\bigoplus_{w \in V}}
      \scL_w \rb \otimes \scL_v^\vee|_Z \rb
  \cong \scH^0 ( \Phi(\scL_v^\vee|_Z))
\end{equation}
where $\scH^0$ denotes the 0-th cohomology functor
with respect to the standard $t$-structure on $D^b \module \bC \Gamma$.
Since $\scH^i(\Phi(\scL_v^\vee|_Z)) = 0$
for $i < 0$,
the above inclusion induces a morphism
$$
 \alpha : S \to \Phi(\scL_v^{\vee}|_Z)
$$
such that the composition
$
 S \to \Phi(\scL_v^{\vee}|_Z)
  \to \Phi(\scL_v^{\vee}|_z)
  \cong E_z
$
is non-zero for any $z \in Z$.
Then $\Phi^{-1}(\alpha)$ is a morphism
$
 \Phi^{-1}(S) \to \scL_v^{\vee}|_Z
$
such that the composition
$
 \Phi^{-1}(S)
  \to \scL_v^{\vee} |_Z
  \to \scL_v^\vee |_z
$
is non-zero for any $z \in Z$.
Since they are both line bundles on $Z$,
the morphism $\Phi^{-1}(\alpha)$ must be an isomorphism.
\end{proof}

\begin{corollary} \label{cr:S0}
For any $v \in R_1$,
one has $\Phi(\scL_v^{\vee}|_Z) \cong S$.
\end{corollary}

Set $\scF = \scL_v(Z)|_Z$, so that we have an exact sequence
$$
 0 \to \scL_v \to \scL_v(Z) \to \scF \to 0.
$$
The dual of the short exact sequence
$$
 0 \to \scL_v^\vee(-Z) \to \scL_v^\vee \to \scL_v^\vee|_Z \to 0
$$
gives the distinguished triangle
$$
 (\scL_v^\vee|_Z)^\vee \to \scL_v \to \scL_v(Z) \to (\scL_v^\vee|_Z)^\vee[1],
$$
which shows that
$
 \scF
  \cong (\scL_v^\vee|_Z)^\vee[1]
  \cong \Phi^{-1}(S)^\vee[1].
$

\begin{corollary} \label{cor:twist}
One has
$$
\bR\Hom(\scF, \scL_w) =
  \begin{cases}
   \bC[-1]  &w \in R_1, \\
   0 &w \in R_2.
  \end{cases}
$$
\end{corollary}

\begin{proof}
\eqref{eq:S} and Corollary \ref{cr:S0} imply
$$
\bR\Gamma((\scL_v^{\vee}|_Z) \otimes \scL_w) =
  \begin{cases}
   \bC  & w \in R_1, \\
   0 & w \in R_2,
  \end{cases}
$$
which together with
$
 \scF \cong (\scL_v^\vee|_Z)^\vee[1]
$
gives the assertion.
\end{proof}

\begin{definition}[{\cite[Definition 1.1]{Seidel-Thomas}}]
An object $\scG$ in a triangulated category
with the trivial Serre functor of degree 3
is a {\em spherical object} if
$$
 \Ext^k(\scG, \scG) =
  \begin{cases}
   \bC & k = 0, 3, \\
   0 & \text{otherwise}. 
  \end{cases}
$$
\end{definition}

\begin{lemma} \label{lm:spherical}
$\scF$ is a spherical object.
\end{lemma}

\begin{proof}
Lemma \ref{lm:simple} gives
$$
\Hom(\scF, \scF)
 \cong \Ext^3(\scF, \scF)^\vee
 \cong \bC,
$$
and the rigidity of $S$ gives
$$
\Ext^1(\scF, \scF)
 \cong \Ext^2(\scF, \scF)^\vee
 \cong 0.
$$
\end{proof}

\begin{lemma} \label{lm:tautological_type0}
Let $\bigoplus_{w \in V} \scL''_w$
be the kernel of the morphism
from $\bigoplus_{w \in V} \scL_w$
to the family $\scQ$
of quotients parametrized by $Z$;
\begin{equation} \label{eq:tautological_type0}
 \bigoplus_{w \in V} \scL''_w
  = \Ker \lb \bigoplus_{w \in V} \scL_w \to \scQ \rb.
\end{equation}
Then the tautological bundle
$\bigoplus_{w \in V} \scL'_w$
on $\scM_{\theta'}$ is given by
\begin{equation} \label{eq:L'}
 \bigoplus_{w \in V} \scL'_w =
  \begin{cases}
   \bigoplus_{w \in V} \scL''_w & v_0 \in R_1, \\
   \bigoplus_{w \in V} \scL''_w(Z) & v_0 \in R_2
  \end{cases}
\end{equation}
where $v_0$ is the fixed vertex such that
$\scL_{v_0} \cong \scL'_{v_0}$ is the trivial line bundle
as in Section \ref{sc:moduli}.
\end{lemma}

\begin{proof}%[proof1]
Consider the following diagram:
\begin{equation} \label{eq:elementary_transform}
\begin{CD}
  @. 0 @. 0 \\
  @. @AAA @AAA \\
 0 @>>> \scS @>>> \bigoplus_{v \in V} \scL_v|_Z @>>> \scQ @>>> 0 \\
 @. @AAA @AAA @| \\
 0 @>>> \bigoplus_{v \in V} \scL''_v
    @>>> \bigoplus_{v \in V} \scL_v @>>> \scQ @>>> 0 \\
  @. @AAA @AAA \\
  @. \bigoplus_{v \in V} \scL_v(-Z) @= \bigoplus_{v \in V} \scL_v(-Z) \\
  @. @AAA @AAA \\
  @. 0 @. 0
\end{CD}
\end{equation}
Here, the row
$$
 0 \to \scS \to \bigoplus_{v \in V} \scL_v|_Z \to \scQ \to 0
$$
is the family of $\theta'$-destabilizing sequences on $Z$.
The diagram \eqref{eq:elementary_transform} gives the following diagram:
\begin{equation} \label{eq:elementary_transform2}
\begin{CD}
  @. @. 0 @. 0 \\
  @. @. @AAA @AAA \\
  @. @. \scS @= \scS \\
 @. @. @AAA @AAA \\
 0 @>>> \bigoplus_{v \in V} \scL''_v(-Z)
   @>>> \bigoplus_{v \in V} \scL''_v
    @>>> \bigoplus_{v \in V} \scL''_v|_Z @>>> 0 \\
  @. @| @AAA @AAA \\
 0 @>>> \bigoplus_{v \in V} \scL''_v(-Z)
  @>>> \bigoplus_{v \in V} \scL_v(-Z)
  @>>> \scQ \otimes \scO(-Z) @>>> 0 \\
  @. @. @AAA @AAA \\
  @. @. 0 @. 0
\end{CD}
\end{equation}
The column
$$
 0
  \to \scQ \otimes \scO(-Z)
  \to \bigoplus_{v \in V} \scL''_v
  \to \scS
  \to 0
$$
gives a family of $\theta$-destabilizing sequences on $Z$.
We show that for every $z \in Z$, the extension
\begin{equation} \label{eq:extension2}
 0
  \to \scQ \otimes \scO(-Z)|_z
  \to \bigoplus_{v \in V} \scL''_v|_z
  \to \scS_z
  \to 0
\end{equation}
is non-split.

Recall that $Z$ is defined by the system of equations corresponding to the arrows
whose sources are in $R_1$ and whose targets are in $R_2$.
In terms of local coordinates in \cite{Ishii-Ueda_08}, these equations are represented
by monomials.
Since $Z$ is a Cartier divisor, $Z$ is locally defined by one of these monomials.
Thus, there are an arrow $a$ with $s(a) \in R_1$, $t(a) \in R_2$ and an open neighborhood $U$ of $z$ such that $\scL_{t(a)}|_U \cong \scL_{s(a)}(Z)|_U$ where the map $\scL_{s(a)} \to \scL_{t(a)}$ over $U$ is identified with the natural inclusion map $\scL_{s(a)}|_U \to \scL_{s(a)}(Z)|_U$.
This shows that $\scL''_{s(a)}|_U \to \scL''_{t(a)}|_U$ is an isomorphism,
proving \eqref{eq:extension2} is non-split.

The non-splitting of \eqref{eq:extension2} shows that $\bigoplus_{v \in V} \scL''_v|_z$
is actually $\theta'$-stable.
Thus $\bigoplus_{v \in V} \scL''$ is a family of $\theta'$-stable representations
parametrized by $\scM_{\theta}$.
Then there is a morphism  $\scM_{\theta} \to \scM_{\theta'}$ which is the identity
outside $Z$.
Therefore this morphism is the identity and the tautological bundles
$\bigoplus_{v \in V} \scL'$ and $\bigoplus_{v \in V} \scL''$ coincides up to a line bundle.
Since we normalized the tautological bundles so that $\scL_{v_0}$ and $\scL'_{v_0}$
are trivial line bundles, we obtain \eqref{eq:L'}.
\end{proof}

For a spherical object $\scG$,
the {\em spherical twist functor}
%acting on objects as
defined by
$$
 T_\scG
  = \Cone \lc \hom(\scG, -) \otimes \scE \xto{\ev} - \rc
  : D^b \coh \scM_\theta \to D^b \coh \scM_\theta.
$$
It is an autoequivalence
of triangulated categories
\cite{Seidel-Thomas}.

\begin{proposition} \label{pr:type0}
%If the conditions $(\bfT)+(\bfE)$ hold
%for the stability parameter $\theta$, then
One has
$
 T_\scF \lb \bigoplus_{v \in V} \scL_v \rb
  \cong\begin{cases}\bigoplus_{v \in V} \scL'_v(Z) & v_0 \in R_1 \\ 
                    \bigoplus_{v \in V} \scL'_v & v_0 \in R_2.
       \end{cases}
$
\end{proposition}

\begin{proof}
This immediately follows
from Corollary \ref{cor:twist} and
Lemma \ref{lm:tautological_type0}.
\end{proof}

Proposition \ref{pr:type0} shows \pref{th:Xi}
with either $\Xi = T_\scF$ or $\Xi=\scO(-Z) \otimes T_\scF$
in the type 0 case
with rigid subobject.

\subsubsection{Rigid quotient case}
 \label{sc:rigid_quotient}

Assume that the quotient object $Q_z$
instead of the subobject $S_z$ is rigid.
The discussion in the rigid subobject case can be adapted
to this case by replacing $\Ext^i(S, E_y)$ with $\Ext^i(E_y, Q)$,
and one shows that $Z$ is a connected Cartier divisor and
\begin{equation} \label{eq:Q0}
 \Phi^{-1}(Q)
  \cong \scL_v^\vee \otimes \omega_Z[2]
\end{equation}
for $v \in R_2$ where $\omega_Z$ is the dualizing sheaf of $Z$ and
$[2]$ is the 2-shift in the derived category.
The object
$\scF = \scL_v|_Z$ is a spherical object and
\pref{th:Xi} holds
with $\Xi = T_\scF^{-1}$ or $\Xi=\scO(Z)\otimes T_\scF^{-1}$.

\subsection{Type I case}

Assume that the wall $W = W_{R_1}$
separating the chambers $C$ and $C'$
%containing $\theta$ and $\theta'$ respectively
is of type I.
Let $\ell \subset \scM_\theta$ be the curve
contracted by $f$.

\begin{lemma} \label{lm:wall1}
One has $\theta(\varphi_C(\scO_\ell)) > 0$
for any $\theta \in C$ and
$\theta_0(\varphi_C(\scO_\ell)) = 0$
for any $\theta_0 \in W$.
\end{lemma}

\begin{proof}
For $\theta \in C$,
one has
$
 \theta(\varphi_C(\scO_\ell))
  = \deg L_C(\theta)|_\ell
$
by \eqref{eq:degree},
which is positive since $L_C(\theta) \in \Pic(\scM_\theta)_\bQ$ is ample.

For $\theta_0 \in W$,
the $\bQ$-line bundle $L_C(\theta_0)$ is the pull-back
of an ample $\bQ$-line bundle in $\Pic(\scMbar_{\theta_0})_\bQ$,
so that $\deg L_C(\theta_0)|_\ell$ cannot be positive
since $\ell$ is contracted by the morphism
$\scM_\theta \to \scMbar_{\theta_0}$.
\end{proof}

\begin{lemma} \label{lm:degree1}
If $v_0 \in R_2$, then $\deg \scL_v|_\ell$ is either $0$ or $1$
for any $v \in V$.
If $v_0 \in R_1$, then $\deg \scL_v|_\ell$ is either $0$ or $-1$
for any $v \in V$.
\end{lemma}

\begin{proof}
One of the defining inequalities for the chamber $C$ is given by
$\theta(\chi_{R_1}) > 0$,
which is also written as $\theta(\varphi_C(\scO_\ell)) > 0$
by Lemma \ref{lm:wall1}.
This implies the existence of a positive rational number $q$
satisfying $\theta(\varphi_C(\scO_\ell)) = q \theta(\chi_{R_1})$
for any $\theta \in \Theta$.
Since $\Theta$ is defined as
$
 \Theta = \{ \theta \in \Hom(\bZ^V, \bZ) \mid \theta(\chi_V) = 0 \},
$
the fact that $\theta (q \chi_{R_1} - \varphi_C(\scO_\ell)) = 0$
for any $\theta \in \Theta$ implies the existence of $k \in \bQ$
such that $q \chi_{R_1} - \varphi_C(\scO_\ell) = k \chi_V$.
By substituting any $v_2 \in R_2$ to this equality,
one shows that
$k = -\varphi_C(\scO_\ell)(v_2)$ is an integer.
Since $\varphi_C([\mathrm{pt}]) = \chi_V$,
one has
$
 \varphi_C(\scO_\ell(k)) = q \chi_{R_1}.
$
Since $q$ is positive and
$\scO_\ell(k)$ is not divisible in $K_0(\scM_\theta)$,
one has $q = 1$, so that
$
 \varphi_C(\scO_\ell(k)) = \chi_{R_1}.
$
Since
$$
 \varphi_C(\scO_\ell(k))(v)
  = \sum_i (-1)^i \dim H^i(\scL_v \otimes \scO_\ell(k))
$$
by the definition of $\varphi_C$,
one has
$$
 \deg (\scL_v \otimes \scO_\ell(k))
  = \begin{cases}
      0 & v \in R_1, \\
      -1 & v \in R_2.
     \end{cases}
$$
Since $\scL_{v_0}$ is the trivial bundle
by our normalization,
one concludes that
$$
 k
  = \deg (\scL_{v_0} \otimes \scO_\ell(k))
  = \begin{cases}
      0 & v_0 \in R_1, \\
      -1 & v_0 \in R_2.
     \end{cases}
$$
\end{proof}

Let $Z \subset \scM_\theta$ be the $\theta'$-unstable locus.

\begin{lemma}
$\ell$ is a connected component of $Z$.
\end{lemma}

\begin{proof}
Since $\deg (\scL_v \otimes \scO_\ell(k))$ is $0$ or $-1$,
the object
$$
 S
  = \Phi_C(\scO_\ell(k))
  = \bR \Gamma \lb \textstyle{\bigoplus_{v \in V}}
      \scL_v \otimes \scO_\ell(k) \rb
$$
in the derived category $D^b \module \bC \Gamma$
is in fact an object of $\module \bC \Gamma$,
which is rigid since $\scO_\ell(k)$ is rigid.
The equivalence $\Phi_C$ gives
$
 \Hom_{\scO_{\scM_\theta}}(\scO_\ell(k), \scO_y)
  \cong \Hom_\Gamma(S, E_y),
$
so that $S$ destabilizes $E_y$
if and only if $y$ is on the rational curve $\ell$.

Since $\theta_0$ is a generic stability parameter
on the wall $W_{R_1}$,
the destabilizing submodule $S_y$ for $y \in Z$
is a $\theta_0$-stable representation
with dimension vector $\chi_{R_1}$.
This gives a morphism from $Z$
to the moduli space of $\theta_0$-stable representations
with dimension vector $\chi_{R_1}$.
Rigidity of $S$ implies that $[S]$ is an isolated point
in this moduli space,
so that $\ell$ is a connected component of $Z$.
\end{proof}

\begin{proposition} \label{pr:Zell}
One has $Z = \ell$.
\end{proposition}

\begin{proof}
Let $Z_c = Z \setminus \ell$ be the complement of $\ell$ in $Z$,
and assume for contradiction that $Z_c$ is non-empty.
Take a connected component $Z_1$ of $Z_c$.
The same argument as in the type 0 case shows that
$Z_1$ is a compact Cartier divisor and
one has either
$\Phi_C(\scL_v^\vee|_{Z_1}) = S$ or
$\Phi_C(\scL_v^\vee \otimes \omega_{Z_1}) = Q$,
so that
$\varphi_C(\scL_v^\vee|_{Z_1}) = \chi_{R_1}$ or
$\varphi_C(\scL_v^\vee \otimes \omega_{Z_1}) = \chi_{R_2}$.
This contradicts
$\varphi_C(\scO_\ell(k)) = \chi_{R_1}$ and
$\varphi_C(\scO_\ell(k-1)) = - \chi_{R_2}$,
and Proposition \ref{pr:Zell} is proved.
\end{proof}

\begin{corollary} \label{cr:strict_transf}
For any $v \in V$,
the tautological bundle
$\scL'_v$ on $\scM_{\theta'}$
is the strict transform of the tautological bundle
$\scL_v$ on $\scM_{\theta}$.
\end{corollary}

\begin{proof}
Since the tautological bundle does not change
outside of the unstable locus,
the fact that the unstable locus $Z = \ell$ is of codimension two implies that
the tautological bundle on $\scM_{\theta'}$
is the strict transform of
the tautological bundle on $\scM_{\theta}$.
\end{proof}

\begin{lemma} \label{lm:flop}
The diagram \eqref{eq:VGIT} is
the Atiyah flop.
\end{lemma}

\begin{proof}
\pref{cr:strict_transf} and \eqref{eq:L_theta1} show that
the strict transform of $\scO_{\scM_\theta}(1) \otimes \scO_{X_{\theta_0}}(1)^\vee$
is $\scO_{\scM_{\theta'}}(-1) \otimes \scO_{X_{\theta_0}}(1)$.
%Consider the divisor $D$
%associated with the line bundle
%$\scO_{\scM_\theta}(1) \otimes \scO_{X_{\theta_0}}(1)^\vee$.
%Then $D$ is $f$-ample and
%$- D$ is $f'$-ample.
Since $f$ is a small primitive contraction of a toric Calabi-Yau 3-fold,
the diagram \eqref{eq:VGIT} is the Atiyah flop.
\end{proof}

Let $\scMtilde = \scM_\theta \times_{X_{\theta_0}} \scM_{\theta'}$
be the fiber product of $\scM_\theta$ and $\scM_{\theta'}$
over $X_{\theta_0}$.
The natural projections will be denoted by
$p : \scMtilde \to \scM_\theta$ and
$q : \scMtilde \to \scM_{\theta'}$.
Both $p$ and $q$ are blow-ups along the unstable loci,
and the exceptional set $E$ is a divisor in $\scMtilde$
isomorphic to $\bP^1 \times \bP^1$.

\begin{theorem}[{Bondal and Orlov \cite{Bondal-Orlov_semiorthogonal}}]
The functor
$$
 \bR q_* \circ \bL p^* : D^b \coh \scM_\theta \to D^b \coh \scM_{\theta'}
$$
is an equivalence of triangulated categories,
whose inverse is given by
$$
 \bR p_* \lb \scO_{\scMtilde}(E) \otimes \bL p^* (-) \rb
  : D^b \coh \scM_{\theta'} \to D^b \coh \scM_{\theta}.
$$
\end{theorem}

\begin{lemma} \label{lm:tautological_type1}
If $v_0 \in R_2$,
then one has
$$
 \bR q_* \circ \bL p^* \lb \scL_v \rb
  \cong \scL_v'
$$
for any $v \in V$.
If $v_0 \in R_1$,
then one has
$$
 \bR q_* \lb \scO_{\scMtilde}(E) \otimes
  \bL p^* \lb \scL_v \rb \rb
  \cong \scL_v'
$$
for any $v \in V$.
\end{lemma}

\begin{proof}
If $v_0 \in R_2$,
then $\deg \scL_v$ is either $0$ or $1$ for any $v \in V$
by Lemma \ref{lm:degree1}, and
$
 \bR q_* \circ \bL p^*(\scL_v) = q_* \lb p^* \scL_v \rb
$
is the strict transform of $\scL_v$.

If $v_0 \in R_1$,
then $\deg \scL_v$ is either $0$ or $-1$ for any $v \in V$
and
$
 \bR q_* \lb \scO_{\scMtilde}(E) \otimes \bL p^* (\scL_v) \rb
  = q_* \lb \scO_{\scMtilde}(E) \otimes p^* \scL_v \rb
$
is the strict transform of $\scL_v$.
\end{proof}

This shows \pref{th:Xi} for type I
with $\Xi =  \bR q_* \circ \bL p^*$ or
$\bR q_* \lb \scO_{\scMtilde}(E) \otimes \bL p^* (-) \rb$.

\subsection{Type {\III} case}

Assume that the wall $W$ separating the chambers $C$ and $C'$
is of type \III,
so that $f : \scM_\theta \to X_{\theta_0}$ contracts
a toric divisor $D \subset \scM_\theta$
to a torus-invariant curve $B \subset X_{\theta_0}$.

\begin{lemma}
The morphism $f' : \scM_{\theta'} \to X_{\theta_0}$
also contracts a divisor $D' \subset \scM_{\theta'}$
to the curve $B \subset X_{\theta_0}$,
and the birational map $\scM_\theta \dashrightarrow \scM_{\theta'}$
associated with the diagram \eqref{eq:VGIT}
extends to an isomorphism.
\end{lemma}

\begin{proof}
The morphism $f'$ is determined uniquely
as the torus-equivariant crepant resolution of $X_{\theta_0}$.
\end{proof}

Let $\ell$ be a curve in $D$
which contracts to a point in $B$.

\begin{lemma} \label{lm:wall3}
One has $\theta(\varphi_C(\scO_\ell)) > 0$
for any $\theta \in C$ and
$\theta_0(\varphi_C(\scO_\ell)) = 0$
for any $\theta_0 \in W$.
\end{lemma}

\begin{proof}
The proof is identical to that of Lemma \ref{lm:wall1}
\end{proof}

\begin{proposition}
One has $Z = D$.
\end{proposition}

\begin{proof}
One can prove that $S = \Phi_C(\scO_\ell(k))$ is an object
of the abelian category just as in the type I case,
where $k = -1$ if $v_0 \in R_2$, and
$k = 0$ if $v_0 \in R_1$.
The object $S$ is a $\theta'$-destabilizing subobject
of $E_z$ for $z \in \ell$.
Unlike the type I case,
the object $S$ is not rigid,
and $B$ is a connected component of the moduli space
of $\theta_0$-stable representations
with dimension vector $\chi_{R_1}$,
so that $D$ is a connected component of $Z$.
The complement $Z \setminus D$ is empty
just as in the type I case.
\end{proof}

\begin{lemma} \label{lm:tautological_type3_s}
If $v_0 \in R_1$, then one has
$$
 \scL'_v =
  \begin{cases}
   \scL_v(-D) & \deg  \scL_v|_\ell = -1, \\
   \scL_v & \deg  \scL_v|_\ell = 0.
  \end{cases}
$$
If $v_0 \in R_2$, then one has
$$
 \scL'_v =
  \begin{cases}
   \scL_v & \deg  \scL_v|_\ell = 0, \\
   \scL_v(D) & \deg  \scL_v|_\ell = 1.
  \end{cases}
$$
\end{lemma}

\begin{proof}
This is proved in the same way
as in Lemma \ref{lm:tautological_type0}.
\end{proof}

Let $\scMtilde = \scM_\theta \times_{X_{\theta_0}} \scM_{\theta'}$
be the fiber product of $\scM_\theta$ and $\scM_{\theta'}$
over $X_{\theta_0}$.
The natural projections will be denoted by
$p : \scMtilde \to \scM_\theta$ and
$q : \scMtilde \to \scM_{\theta'}$.
%Both $p$ and $q$ are blow-ups along the unstable loci,
%and the exceptional set $E$ is a divisor in $\scMtilde$
%isomorphic to $\bP^1 \times \bP^1$.
Let $\omega_{\scMtilde}$ be the dualizing sheaf of $\scMtilde$.
The scheme $\scMtilde$ has two irreducible components;
one is isomorphic to $\scM:=\scM_\theta \cong \scM_{\theta'}$ and the other
is $D \times _B D$.
One has exact sequences
\begin{equation}\label{eq:scMtilde1}
0 \to \scO_{\scM}(-D) \to \scO_{\scMtilde} \to \scO_{D \times_B D} \to 0
\end{equation}
and
\begin{equation}\label{eq:scMtilde2}
0 \to \omega_{\scM} \to \omega_{\scMtilde} \to \omega_{D \times_B D}(\Delta_D) \to 0,
\end{equation}
where $\Delta_D \subset D \times_B D$ is the diagonal.

\begin{lemma} \label{lm:tautological_type3_q}
If $v_0 \in R_1$,
then one has
$$
 \bR q_* \circ \bL p^* \lb \scL_v \rb
  \cong \scL_v'
$$
for any $v \in V$.
If $v_0 \in R_2$,
then one has
$$
 \bR q_* \lb \omega_{\scMtilde} \otimes
  \bL p^* \lb \scL_v \rb \rb
  \cong \scL_v'
$$
for any $v \in V$.
\end{lemma}

\begin{proof}
This is proved by using Lemma \ref{lm:tautological_type3_s},
\eqref{eq:scMtilde1} and \eqref{eq:scMtilde2}
\end{proof}

\begin{theorem}%[{Horja \cite{Horja_DCAMS}}]
The functor
$$
 \bR q_* \circ \bL p^* : D^b \coh \scM_\theta \to D^b \coh \scM_{\theta'}
$$
is an equivalence of triangulated categories,
whose inverse is given by
$$
 \bR p_* \lb \omega_{\scMtilde} \otimes \bL p^* (-) \rb
  : D^b \coh \scM_{\theta'} \to D^b \coh \scM_{\theta}.
$$
\end{theorem}
\begin{proof}
These functors are compositions of Horja's EZ-transforms \cite{Horja_DCAMS}
and the tensor products by  line bundles as explained
in \cite[Remark 7.6]{Craw-Ishii}.
\end{proof}

This shows \pref{th:Xi} for type \III \ 
with $\Xi = \bR q_* \circ \bL p^*$
or $\bR p_* \lb \omega_{\scMtilde} \otimes \bL p^* (-) \rb$.

\section{Projective crepant resolutions as moduli spaces}
 \label{sc:crepant_resolution}

In this section,
we give a proof of Corollary \ref{cr:crepant_resolution}.
Since the proof is completely parallel
to that of the main theorem of \cite{Craw-Ishii},
we will be brief here and
refer the reader to \cite{Craw-Ishii} for details.

Let $Y = \scM_\theta$ be the moduli space
for some generic stability parameter $\theta \in \Theta$,
and $X = X_0$ be the normalization
of the irreducible component of the moduli space $\scMbar_0$
for the stability parameter $0 \in \Theta$
containing the torus $\bT$.
According to \cite[Proposition 6.3]{Ishii-Ueda_08},
the space $X$ is the Gorenstein affine toric variety
associated with the characteristic polygon
of the dimer model $G$.
Since $Y$ is a toric variety,
its ample cone $\Amp(Y)$
is an interior of a convex rational polyhedral cone.
Since flop is an isomorphism in codimension one,
one can identify the Picard group
of any crepant resolutions.
The union
$$
 \overline{\Mov(Y)} = \bigcup_{Y'} \overline{\Amp(Y')}
  \subset \Pic(Y)_\bQ
$$
of the nef cones
over the set of projective crepant resolutions $Y'$
is the closure of the {\em movable cone}.
Ample cones of birational Calabi-Yau 3-folds
give a polyhedral decomposition
of the movable cone
\cite{Kawamata_CBU}.

Consider the diagram
$$
\begin{CD}
 \Theta_\bQ = \bigcup \overline{C} @>{\varphi^*}>> F^1 \\
  @V{L}VV @VV{p}V \\
 \Pic(Y)_\bQ @<{\sim}<< F^1/F^2.
\end{CD}
$$
The space $\Theta_\bQ$ is divided into chambers
by walls of types 0, I, and \III.
The map $L : \Theta_\bQ \to \Pic(Y)_\bQ$ is
a piecewise-linear map
which is given by $L_C$
on each chamber $C \subset \Theta_\bQ$.
The map $\varphi^*$ is defined similarly
by the collection $\varphi_C^* : C \to F^1$
of linear maps.
The image of $L$ sits inside
the movable cone $\Mov(Y)$,
and one asks whether $L$ surjects onto $\Mov(Y)$.
Lemma \ref{lm:flop} shows that
the moduli space flops
across a wall $W \subset \Theta_\bQ$ of type I.
Since any crepant resolution
$Y' \to X$ is related to $Y$
by a sequence of flops,
it suffices to show the following
to prove Corollary \ref{cr:crepant_resolution}:

\begin{proposition} \label{pr:flop}
Let $Y = \scM_\theta$ be the moduli space
for a generic stability parameter $\theta \in \Theta$.
For any wall $\Wbar \subset \Pic(Y)_\bQ$ of type I
on the boundary of the nef cone $\overline{\Amp(Y)}$,
there is a chamber $C \subset \Theta$
and a wall $W$ of the chamber $C$
such that $L_C(W) = \Wbar$.
\end{proposition}

Although the chamber $C$ may not be the chamber $C_\theta$
containing the stability parameter $\theta$,
one can reach the chamber
containing the desired wall $W$ of type $I$ on the boundary
by crossing a finite number of walls of type 0.
To show this,
it is natural to consider the union
$\bigcup_{\scM_C \cong Y} \varphi_C^*(\Cbar)$
of closures $\Cbar$ of chambers such that
$\scM_C$ is isomorphic to $Y$,
where $\scM_C$ is the moduli space $\scM_\theta$
for one (and hence any) stability parameter $\theta \in C$.

Unfortunately,
this is still not enough,
since adjacent chambers in $\Theta$ may not be sent
to adjacent cones in $F^1$.
This comes from the ambiguity
in the choice of the tautological bundle,
which we have fixed by hand
by choosing a vertex $v_0 \in V$ and
setting $\scL_{v_0} = \scO_{\scM_\theta}$.
To solve this problem,
%Craw and Ishii
\cite{Craw-Ishii} considers the union
$
 \bigcup_{L \in \Pic^c(Y)} \bigcup_{\scM_C \cong Y}
  L \otimes \varphi_C^*(\Cbar)
$
over the subgroup $\Pic^c(Y)$ of $\Pic(Y)$
generated by $\scO_Y(S)$ for compact divisors $S$.
The reason that it suffices to consider only compact divisors
comes from the fact that the unstable locus
for a type 0 wall is a compact divisor,
so that the variation of the tautological bundle
given in Lemma \ref{lm:tautological_type0}
is described solely in terms of a compact divisor.
Let $\Amp'(Y) \subset \overline{\Amp(Y)}$
be the complement of the walls of type \II.
The advantage of working with $\Pic^c(Y)$
instead of $\Pic(Y)$ lies in the following lemma:

\begin{lemma}[{\cite[Lemma 8.1]{Craw-Ishii}}]
 \label{lm:Craw-Ishii_8.1}
 Let $S_1, \dots, S_b$ be a basis of $\Pic^c(Y)$.
If $p(\xi) \in \Amp'(Y)$ for $\xi \in F^1$,
then $[\scO_{S_1} \otimes \xi], \dots, [\scO_{S_b} \otimes \xi]$
are linearly independent elements in $F^2$.
\end{lemma}

To prove Proposition \ref{pr:flop},
one uses the following description of the chamber $C$:

\begin{theorem}[{\cite[Theorem 5.9]{Craw-Ishii}}]
 \label{th:Craw-Ishii_5.9}
%Let $C \subset \Theta_\bR$ be a chamber.
One has $\theta \in C$ if and only if
\begin{itemize}
 \item
for every exceptional curve $\ell$, one has
$
 \theta(\varphi_C(\scO_\ell)) > 0,
$
and
 \item
for every compact reduced divisor $D$
and every vertex $v \in V$,
one has
\begin{equation} \label{eq:ineq_0}
 \theta(\varphi_C(\scL_v^\vee \otimes \omega_D)) < 0
  \quad \text{and} \quad
 \theta(\varphi_C(\scL_v^\vee|_D)) > 0.
\end{equation}
\end{itemize}
\end{theorem}

\begin{proof}[Sketch of proof]
The inequality arising from a wall of type I or \III \ 
is of the form
$
 \theta(\varphi_C(\scO_\ell)) > 0
$
by Lemma \ref{lm:wall1} and
Lemma \ref{lm:wall3}.
The inequality coming from a wall of type 0 is of the form
$
 \theta(\varphi_C(\scL_v^\vee|_D)) > 0
$
in the case when the unstable locus $D$ parameterizes rigid submodules
by Corollary \ref{cr:S0} and
$\theta(\dim S) > 0$,
and of the form
$
 \theta(\varphi_C(\scL_v^\vee \otimes \omega_D)) < 0
$
in the case when $D$ parameterizes rigid quotients
by \eqref{eq:Q0} and
$\theta(\dim Q) < 0$.
Conversely, any parameter $\theta$ in $C$ 
has to satisfy the inequalities \eqref{eq:ineq_0}
for any compact reduced divisor
%follows from these cases 
by the same argument
as in the proof of \cite[Lemma 5.7]{Craw-Ishii}.
\end{proof}

The following proposition
concludes the proof of Proposition \ref{pr:flop}:

\begin{proposition}[{\cite[Proposition 8.2]{Craw-Ishii}}]
 \label{pr:Craw-Ishii_8.2}
For every $\xi \in p^{-1}(\Amp'(Y))$,
there is a neighborhood $N(\xi)$
of $\xi$ in $F^1$ such that only finitely-many pairs $(L, C)$
of $L \in \Pic^c(Y)$ and chamber $C \subset \Theta$ satisfy
$
 N(\xi) \cap L \otimes \varphi_C^*(\Cbar) \ne \emptyset.
$
Moreover, one has
$$
 p^{-1}(\Amp'(Y)) \subset
  \bigcup_{L \in \Pic^c(Y)} \bigcup_{\scM_C \cong Y}
  L \otimes \varphi_C^*(\Cbar).
$$
\end{proposition}

\begin{proof}[Sketch of proof]
The first statement is a local finiteness,
which comes from Lemma \ref{lm:Craw-Ishii_8.1} and
Theorem \ref{th:Craw-Ishii_5.9};
Theorem \ref{th:Craw-Ishii_5.9} shows that
$\varphi_C^*(C)$ is bounded
in the fiber direction of $p$,
and Lemma \ref{lm:Craw-Ishii_8.1} shows that
taking the union over $\Pic^c(Y)$ does not destroy
the local finiteness.

The second statement essentially comes
from the fact that the derived equivalence
caused by crossing a wall of type 0
is a spherical twist,
which does not change the orientation
in the 3-dimensional case.
It follows that if a chamber $C'$ is on the other side
of the wall $W$ from the chamber $C$,
then $\varphi_{C'}^*(C')$ is on the other side
of $\varphi_C^*(W)$
from $\varphi_C(C)$
up to tensor by a line bundle $L \in \Pic^c(Y)$
associated with a compact divisor.

The detail of the proof
%of Proposition \ref{pr:Craw-Ishii_8.2}
is identical to the one in \cite{Craw-Ishii},
and we omit it here.
\end{proof}

Proposition \ref{pr:flop} is an immediate consequence
of Proposition \ref{pr:Craw-Ishii_8.2};
the second statement ensures that
for any wall $\Wbar \subset \Pic(Y)_\bQ$ of type I,
one can go arbitrarily close to that wall
by moving in the space of stability parameters
and tensoring line bundles $L \in \Pic^c(Y)$.
The local finiteness in the first statement
ensures that one in fact can find finitely many walls of type 0
such that by crossing them,
one can reach the desired wall.

\bibliographystyle{amsalpha}
\bibliography{bibs}

\noindent
Akira Ishii

Department of Mathematics,
Graduate School of Science,
Hiroshima University,
1-3-1 Kagamiyama,
Higashi-Hiroshima,
739-8526,
Japan

{\em e-mail address}\ : \ akira@math.sci.hiroshima-u.ac.jp

\ \\

\noindent
Kazushi Ueda

Department of Mathematics,
Graduate School of Science,
Osaka University,
Machikaneyama 1-1,
Toyonaka,
Osaka,
560-0043,
Japan.

{\em e-mail address}\ : \  kazushi@math.sci.osaka-u.ac.jp

\end{document}